\documentclass{amsproc}
\usepackage{amscd,amsmath,amssymb,amsfonts,verbatim}
\usepackage[all]{xy}
\usepackage{comment}

\newtheorem{thm}{Theorem}[section]
\newtheorem{prop}[thm]{Proposition}
\newtheorem{lem}[thm]{Lemma}
\newtheorem{cor}[thm]{Corollary}

\theoremstyle{definition}
\newtheorem{ques}[thm]{Question}

\newtheorem{defn}[thm]{Definition}
\theoremstyle{remark}
\newtheorem{remk}[thm]{Remark}
\newtheorem{remks}[thm]{Remarks}

\newtheorem{exm}[thm]{Example}
\newtheorem{exms}[thm]{Examples}
\newtheorem{notat}[thm]{Notation}
\numberwithin{equation}{section}

{\hfill$\square$\end{defn}}
{\hfill$\square$\end{remk}}
{\hfill$\square$\end{remks}}
{\hfill$\square$\end{exm}}
{\hfill$\square$\end{exms}}
{\hfill$\square$\end{notat}}

\newcommand{\thmref}{Theorem~\ref}
\newcommand{\propref}{Proposition~\ref}
\newcommand{\corref}{Corollary~\ref}

\newcommand{\lemref}{Lemma~\ref}

\newcommand{\sC}{{\mathcal C}}

\newcommand{\sI}{{\mathcal I}}

\newcommand{\sK}{{\mathcal K}}

\newcommand{\sO}{{\mathcal O}}

\newcommand{\sR}{{\mathcal R}}

\newcommand{\sZ}{{\mathcal Z}}
\newcommand{\A}{{\mathbb A}}

\newcommand{\C}{{\mathbb C}}

\newcommand{\F}{{\mathbb F}}

\newcommand{\Q}{{\mathbb Q}}
\newcommand{\R}{{\mathbb R}}

\newcommand{\Z}{{\mathbb Z}}

\newcommand{\fm}{{\mathfrak m}}

\newcommand{\fp}{{\mathfrak p}}
\newcommand{\fq}{{\mathfrak q}}

\newcommand{\CH}{{\rm CH}}

\newcommand{\surj}{\twoheadrightarrow}
\newcommand{\inj}{\hookrightarrow}
\newcommand{\red}{{\rm red}}

\newcommand{\Pic}{{\rm Pic}}

\newcommand{\Spec}{{\rm Spec \,}}

\newcommand{\divf}{{\rm div}}

\newcommand{\Sch}{{\operatorname{\mathbf{Sch}}}}

\newcommand{\<}{\langle}
\renewcommand{\>}{\rangle}

\newcommand{\Sm}{{\mathbf{Sm}}}

\newcommand{\ds}{{/\kern-3pt/}}

\newcommand{\colim}{\mathop{\text{\rm colim}}}

\newcommand{\ov}{\overline}

\renewcommand{\dim}{\text{\rm dim}}

\newcommand{\tuborg}{\left\{\begin{array}{ll}}
\newcommand{\sluttuborg}{\end{array}\right.}

\newcommand{\zar}{{\rm zar}}

\newcommand{\rk}{{\rm rk}}
\newcommand{\gp}{{\rm gp}}
\newcommand{\wt}{\widetilde}
\newcommand{\wh}{\widehat}
\newcommand{\Int}{{\rm Int}}

\def\ol#1{\overline{#1}}

\newcounter{elno}

\newcounter{elno-abc}   

\newcounter{elno-abc-prime}

\begin{document}
\title{Negative $K$-theory and Chow group of monoid algebras}
\author{Amalendu Krishna, Husney Parvez Sarwar}
\address{School of Mathematics, Tata Institute of Fundamental Research,  
1 Homi Bhabha Road, Colaba, Mumbai, India}
\email{amal@math.tifr.res.in}
\email{mathparvez@gmail.com}

%\dedicatory{}



%\subjclass[2010]{Primary 14C25; Secondary 13F35, 14F30, 19E15}

\begin{abstract}  
We show, for a finitely generated partially cancellative torsion-free
commutative monoid $M$, that $K_i(R) \cong K_i(R[M])$ whenever 
$i \le -d$ and $R$ is a quasi-excellent $\Q$-algebra 
of Krull dimension $d \ge 1$. 
In particular, $K_i(R[M]) = 0$ for $i < -d$. This is a generalization of
Weibel's $K$-dimension 
conjecture to monoid algebras. We show that this generalization fails for
$X[M]$ if $X$ is not an affine scheme. We also show that the Levine-Weibel
Chow group of 0-cycles $\CH^{LW}_0(k[M])$ vanishes for any finitely generated 
commutative partially cancellative monoid $M$ if $k$ is an algebraically 
closed field.
\end{abstract}

\maketitle

%\setcounter{tocdepth}{1}

%\tableofcontents  

\section*{Introduction}\label{sec:Intro}
The monoid algebras are natural generalizations of polynomial and
Laurent polynomial algebras over commutative rings. One can often think of them 
as subalgebras of polynomial or Laurent polynomial algebras generated
by monomials. These are as ubiquitous in the study of various properties
of rings as the  polynomial or Laurent polynomial algebras.
A very natural question in algebraic $K$-theory is to find out to what extent 
various known facts about the $K$-theory of polynomial and Laurent 
polynomial algebras remain valid for more general monoid algebras.

Gubeladze proved several results on this subject in a series of many papers
(\cite{Gub-0}, \cite{Gub-3}, \cite{Gub-4} and \cite{Gub-5} to name a few).
Using the new direction provided by \cite{Haes} and \cite{CHSW}, 
Corti{\~n}as, Haesemeyer, Walker and Weibel together have made significant 
advances in the study of algebraic $K$-theory of monoid algebras
(see \cite{CHW-1}, \cite{CHW-2}, \cite{CHW-3} and \cite{CHW-4}). An old
question on the $K$-theory of monoid algebras was recently settled in 
\cite{KS}. Gubeladze recently settled an old $K$-theoretic question about
monoid algebras \cite{Gub-7}.
The message that comes out of these papers is that even
though some properties of the algebraic $K$-theory of polynomial and 
Laurent polynomial algebras remain valid for monoid algebras, many of them
do not directly extend. 

This paper was motivated by our pursuit for those properties of  
the algebraic $K$-theory of polynomial and Laurent polynomial algebras
which may extend to monoid algebras. Our intuition was that while the higher
Quillen $K$-theory of monoid algebras may not resemble 
that of polynomial and Laurent polynomial algebras, the
negative $K$-theory and part of $K_0$ should. 
In this paper, we attempt to understand two such
properties, namely, Weibel's $K$-dimension conjecture and, vanishing of $SK_0$ and
Levine-Weibel Chow group for
monoid algebras. 

\subsection{Weibel's conjecture for monoid algebras}
Recall that a famous conjecture of Weibel \cite{Weibel} asserts that if $R$ is 
a commutative Noetherian
ring of Krull dimension $d$, then $K_{-d}(R) \simeq K_{-d}(R[t_1, \cdots , t_n])$
and $K_i(R[t_1, \cdots , t_n]) = 0$ for $i < -d$ and  $n \ge 0$. 
An affirmative answer to this conjecture was obtained recently by
Kerz, Strunk and Tamme \cite{KST16}. For Noetherian rings containing $\Q$,
this was earlier solved by Corti{\~n}as, Haesemeyer, Schlichting and
Weibel \cite{CHSW} (see also \cite{GH}, \cite{Krishna} and \cite{Weibel-1}
for older results in positive characteristics).

The main technical tool that goes into the proof of
Weibel's conjecture in \cite{KST16} is a pro-cdh-descent theorem
for algebraic $K$-theory. However, the final step in the proof of
the conjecture using the pro-cdh-descent breaks down in the case of
general monoids because it crucially uses the fact that the map
$R \to R[M]$ is smooth.

In this paper, we shall use a combination of pro-cdh-descent techniques 
and the theory of monoid algebras to show that
the assertion of Weibel's conjecture directly extends to more general monoid 
algebras over some rings. 
Broadly speaking, we prove the following. We refer to \S~\ref{sec:MR} for
the precise results including the relevant terms and notations.

\begin{thm}\label{thm:MT-1}
Let $M$ be a monoid which is finitely generated, commutative, 
partially cancellative and torsion-free. Let $R$ be a quasi-excellent
$\Q$-algebra of Krull dimension $d$.
Assume that one of the following conditions is satisfied.
\begin{enumerate}
\item
$d \ge 1$.
\item
$M$ is cancellative and semi-normal.
\end{enumerate}

Then the map $K_i(R) \to K_i(R[M])$ is an isomorphism for all
$i \le -d$. In particular, $K_i(R[M]) = 0$ for all $i < -d$.
\end{thm}

One can interpret this result as a generalization of Weibel's
$K_{-d}$-regularity and $K_{< -d}$-vanishing conjectures to monoid algebras.

\begin{remk}\label{remk:d=0}
We remark that the condition (2) of \thmref{thm:MT-1} is essential
if $d =0$. In this case, we can in fact assume that $R$ is 
a field (see \lemref{lem:affred}).
Assuming this, Gubeladze has shown that $K_0(R) \cong K_0(R[M])$ if and
only if $M$ is semi-normal (see \cite[Theorem~1.3]{Gub-4}).
\end{remk}

Even if we expect \thmref{thm:MT-1} to hold for $\F_p$-algebras, 
we do not know how to prove it yet.
We hope to pursue this in a future work. However, if we allow ourselves
to invert the characteristic of the ground field, then we can indeed
show that $K_i(R) \cong K_i(R[M])$ for all $i \le -d$ when $M$ is as in 
\thmref{thm:MT-1}. A proof of this is given in \S~\ref{sec:MR}.

\vskip .3cm

For non-affine schemes, we can prove the following extension of
Weibel's conjecture.

\begin{thm}\label{thm:scheme-main-vanish-Intro}
Let $M$ be a cancellative torsion-free monoid and $X$ a quasi-excellent 
$\Q$-scheme of dimension $d$. Then $K_i(X[M]) = 0$ for $i < -d$.
If $M$ is semi-normal, then the map $K_i(X) \to K_i(X[M])$ is an isomorphism 
for all $i \le -d$.  
\end{thm}

\vskip .3cm

If $M$ is not semi-normal in \thmref{thm:scheme-main-vanish-Intro}, then
we can show that the monoid extension of Weibel's 
$K_{-d}$-regularity conjecture to general monoids is not valid for 
non-affine schemes.
We deduce this failure from the following precise result.

\begin{thm}\label{thm:Counter**}
Let $X$ be a connected smooth projective curve over a field $k$ of
characteristic zero. Let $M \subset \Z_+$ be the submonoid generated by
$\{2,3\}$.
Then the map $K_{-1}(X) \to K_{-1}(X[M])$ is not an isomorphism if
the genus of $X$ is positive.
\end{thm}

\subsection{Levine-Weibel Chow group and $SK_0$ of monoid algebras}
Apart from the negative $K$-theory, we also wanted to look at other
cohomological properties of polynomial and Laurent polynomial algebras
which can generalize to monoid algebras and which are related to the
non-negative $K$-theory. One of these cohomology groups is the
Levine-Weibel Chow group of 0-cycles $\CH^{LW}_0(X)$ 
for a singular variety $X$ \cite{LW}. 
This group is the singular analogue of the classical Chow group of 0-cycles on
smooth varieties and it is directly related to the $K$-theory of locally
free sheaves.

It is well known that the Levine-Weibel Chow group
vanishes for polynomial and Laurent polynomial algebras over a field.
On the other hand, even though affine toric varieties come up very naturally in
algebraic geometry, it was not known yet if their
Levine-Weibel Chow group is zero in positive characteristics.
Our next result of this paper is the following.
%Using the techniques of Gubeladze \cite[Theorem~1.3]{Gub-7} 
%and the affine Roitman torsion theorem \cite{Krishna-1}, we
%prove the following.

\begin{thm}\label{thm:MT-2}
Let $k$ be an algebraically closed field and $M$ a finitely 
generated commutative partially cancellative monoid of rank at least two.
Then $\CH^{LW}_0(k[M]) = 0$.
\end{thm}

Note that the lower bound on the rank of the monoid is essential because
it is well known that $\CH^{LW}_0(k[t^2, t^3]) \cong k$.
Combined with \cite[Corollaries~2.7 and 3.4]{Murthy},
the vanishing of the Levine-Weibel Chow group for monoid algebras
has the following algebraic consequence.

\begin{cor}\label{cor:MT-3}
Let $k$ be an algebraically closed field and $M$ a finitely 
generated commutative partially cancellative monoid of rank $n \ge 2$. 
Then every local complete 
intersection ideal of $k[M]$ of height $n$ is a complete intersection.
\end{cor}

\vskip .3cm

One of the two main ingredients in our proof of \thmref{thm:MT-2} is
the following result of independent interest in the algebraic $K$-theory
of monoid algebras. When the base ring is regular and the underlying monoid
is cancellative and torsion-free, this result is due to Gubeladze
\cite[Theorem~1.3]{Gub-4}. 
%Using some very recent techniques of Gubeladze \cite{Gub-7}, we can prove the 
%following generalization of \cite{Gub-4}.

\begin{thm}\label{thm:MT-4}
Let $R$ be a commutative Artinian ring and $M$ a finitely generated
commutative partially cancellative monoid. Then $SK_0(R[M]) = 0$.
\end{thm}

One may recall that in most of Gubeladze's works (except \cite{Gub-7})
on monoid algebras, all monoids 
were cancellative and torsion-free. However, it turns out that
for computing the $K$-theory of cancellative monoid algebras
using various fundamental results like excision in algebraic
$K$-theory, one needs to extend many $K$-theoretic results to
partially cancellative torsion-free monoids. These monoids occur very
naturally in affine geometry. The algebras over such monoids
were first studied by Swan 
\cite[\S~15]{Swan} under the name of discrete Hodge algebras. 
Higher $K$-theory analogue of Serre's problem on
projective modules for such algebras was studied by Vorst \cite{Vorst-1} and 
Gubeladze \cite{Gub-2}.
More $K$-theoretic importance of these monoid algebras 
was later exhibited in several papers of
Corti{\~n}as, Haesemeyer, Walker and Weibel. See, for
instance, \cite{CHW-3} and \cite{CHW-4}.

\begin{remk}\label{remk:Counter-ex}
The reader may have observed that the class of monoids considered in 
\thmref{thm:MT-1} is more restrictive than the one in
\thmref{thm:MT-2} in that we did not allow torsion monoids in
\thmref{thm:MT-1}. The reason for this is the following.
If we let $M$ be a finite torsion group of order $n \ge 2$  and 
$R$ a quasi-excellent $\Q$-algebra of dimension $d \ge 1$ such that 
$K_{-d}(R) \neq 0$ (for instance, take $R$ to be the coordinate ring of
a product of affine nodal curves over $\C$), 
then we know that $K_{-d}(R[M]) = K_{-d}(R[M]_{\rm red})$ contains at least
two copies of $K_{-d}(R)$ as direct summands one of which is the
inclusion $K_{-d}(R) \subset K_{-d}(R[M])$. In particular,
the canonical map $K_{-d}(R) \to K_{-d}(R[M])$ is not an isomorphism.
This shows that the $i = -d$ case  \thmref{thm:MT-1} fails for torsion 
monoids. We already saw that the map $K_{-d}(X) \to K_{-d}(X[M])$ is not an 
isomorphism even when $M$ is torsion-free if $X$ is not affine.

However, if we let $i < -d$, then the situation is likely to be
different as some of the above results show.
We expect this case of \thmref{thm:MT-1} to hold for a very
general class of monoids and a general class of Noetherian schemes. 
So we end the discussion of our results with the following.
%For any scheme $X$, let $X[M] = X \times_{\Spec({\Z})} \Spec(\Z[M])$.
\end{remk}

\begin{ques}\label{question-*}
Let $X$ be a Noetherian separated scheme of Krull dimension $d \ge 0$ and
$M$ a finitely generated commutative cancellative monoid. Is
$K_i(X[M]) = 0$ for $i < -d$?
\end{ques}

\vskip .2cm

\subsection{Outline of proofs}\label{sec:Outline}
We give a brief outline of our proofs. Our main strategy for proving
\thmref{thm:MT-1} is to eventually reduce the proof
to the case of cancellative torsion-free semi-normal monoids with the help
of several reductions. To prove
the theorem in this restrictive case, we need to use the pro-cdh descent results
of \cite{KST16}. 
The pro-cdh-descent results and the weak resolution of singularities 
together allow us to reduce to the case when the base scheme $X$ is regular. 
The proof of the vanishing in this case is done using some classical
results of Gubeladze and a Zariski descent argument.
In \S~\ref{sec:prelim}, we recall the basic results about monoids and
prove our reduction steps that are needed in the proof of \thmref{thm:MT-1}.
Gubeladze's Milnor square for monoid algebras plays a key role in these
steps.
The final proofs of Theorems~\ref{thm:MT-1} and 
~\ref{thm:scheme-main-vanish-Intro} using pro-cdh descent are
given in \S~\ref{sec:MR}.

The idea of proving \thmref{thm:MT-2} came to us from two known results.
The first is a classical result of Gubeladze which says that
$SK_0(k[M])$ vanishes if $k$ is a field and $M$ is cancellative torsion-free.
The second is an old conjecture of Murthy \cite{Murthy} which says
that the Levine-Weibel Chow group of an affine variety over an 
algebraically closed field is torsion-free. This conjecture was recently
settled in \cite{Krishna-1}. In view of this positive answer to Murthy's
conjecture, we are left with proving \thmref{thm:MT-4} which extends
Gubeladze's result to a more general class of monoids.

We prove this extension in Sections~\ref{sec:Redn} and ~\ref{sec:LW-**}.
For the proof of \thmref{thm:MT-4}, we have to establish several
reduction steps to reduce the proof to a nicer class of monoids.
This is done in \S~\ref{sec:Redn} using several Milnor squares.
A crucial part of our 
reduction steps is a recent technique of Gubeladze \cite{Gub-7}, which 
tells us how one should deal with cancellative torsion monoids.
This reduction technique plays a key role in our proofs.
The final proof is obtained in \S~\ref{sec:LW-**}. 
The key step in the final proof is a generalization of Swan's result 
\cite[Lemma~15.6]{Swan} on the prime decomposition of radical 
ideals in cancellative torsion-free monoids to 
a more general class of monoids. This allows us to reduce the case of
partially cancellative monoids to cancellative monoids.

One could now naturally ask if Theorems~\ref{thm:MT-2} and ~\ref{thm:MT-4}
could be valid for non-cancellative and other more general class of monoids.
In the hope of dealing with this question in future, we
came up with one more proof of \thmref{thm:MT-4} for partially cancellative 
torsion-free monoids. This proof is given in the end of \S~\ref{sec:LW-**}.
It is substantially $K$-theoretic in flavor and hence has 
potential to generalize to a broader class of monoids.

\begin{comment}
Our approach of proving the vanishing of the Levine-Weibel Chow group
for a monoid algebra $k[M]$ is to relate this group with a certain subgroup
of $SK_0(k[M])$ using the affine Roitman torsion theorem \cite{Krishna-1}.
Gubeladze \cite[Theorem~1.3]{Gub-4} had shown that $SK_0(k[M]) = 0$
if $M$ is finitely generated, commutative and torsion-free monoid.
In this paper, we generalize Gubeladze's result to include the
partially cancellative monoids as well. We thus prove:

\begin{thm}\label{thm:MT-4}
Let $k$ be any field and $M$ any monoid which is finitely generated, 
commutative, partially cancellative and torsion-free. Then the following hold.
\begin{enumerate}
\item
$SK_0(k[M]) = 0$. 
\item
$\CH^{LW}_0(k[M]) = 0$.
\end{enumerate}
\end{thm}
\end{comment}

\section{Recollection and reduction steps}\label{sec:prelim}
In this section, we fix our notations and provide a limited recollection
of the basic definitions in the theory of monoids.
We then establish our reduction steps for reducing the proof of
\thmref{thm:MT-1} to the case of positive cancellative torsion-free 
semi-normal monoids.

Recall that a monoid is same as a semigroup with an identity element.
Throughout this paper, we shall assume all monoids to be commutative 
and finitely generated. We shall assume all our rings to be commutative and
Noetherian and all schemes to be separated and Noetherian.
The dimension of a ring or a scheme will be its Krull dimension.

\subsection{Recollection of monoids}\label{sec:Monoids}
Unless specified otherwise, we shall use the additive notation
for the operation in a monoid but switch to multiplicative operation 
when we consider a monoid inside an algebra generated by it over a commutative
ring. 

For a monoid $M$, we shall let $\gp (M)$ denote the group 
completion of $M$. We have a natural monoid homomorphism
$M \to \gp(M)$. We let $U(M)$ denote the largest submonoid of $M$ which is
also a group. We say that $M$ is positive if $U(M) = 0$.
The rank of the monoid $M$ is the rank of free part of $\gp(M)$.
We shall denote it by $\rk(M)$.
Given a subset $S \subset M$, we shall let $\<S\>$ denote the
submonoid of $M$ generated by $S$. It is the $\Z_+$-linear combination of
elements of $S$.

Given a monoid $M$, we can construct another monoid $M_+$ by adding
a base-point $\infty$ to $M$ and by letting $m + \infty = \infty$ for
all $m \in M_+$.  
We shall call $M_+$ the augmented monoid associated to $M$.
It is a pointed monoid in the language of \cite{CHW-3}.
We have a canonical inclusion of monoids $M \subset M_+$. We shall usually
avoid the usage of more general pointed monoids considered in \cite{CHW-3} as 
we have no need for them. Every monoid homomorphism $f: M \to N$ uniquely 
extends to a monoid homomorphism $f: M_+ \to N_+$ which fixes the base-point.

A monoid $M$ is called cancellative if for $a,b,c \in M$, the condition
$a+b = a+c$ implies that $b = c$. This is equivalent to saying
that the group completion map $M \to \gp(M)$ is injective.
We say that $M$ is torsion-free if
for $a, b \in M$, the condition $na = nb$ for some $n \ge 1$ implies that
$a = b$. If $M$ is cancellative, then it is torsion-free if and only if
$\gp(M) \simeq \Z^r$.

A subset $I \subset M$ is called an ideal in $M$ if $I + M = I$. 
If $I \subset M$ is a proper ideal ($I \neq M$), then $I \cap U(M) = \emptyset$
and $M \setminus U(M)$ is the largest proper ideal of $M$.
If $I \subset M$ is an ideal,
we shall let $I_*$ denote $I \cup \{0\}$. This is a submonoid of $M$.

If $I \subset M$ is an ideal, the quotient $M/I$ is obtained as follows.
We consider the inclusion $I \subset M_+$ and then take the quotient
by the equivalence relation $m \sim \infty$ for all $m \in I$.
There is a unique addition rule in ${M_+}/I$ that turns the canonical 
surjection $M_+ \to {M_+}/I$ into a morphism of monoids
(see \cite[\S~1]{CHW-3}). 
We let $M/I$ be the image of $M$ under the quotient map
$M_+ \to {M_+}/I$. There is an epimorphism of monoids $M \to M/I$ and
a canonical isomorphism $(M/I)_+ \simeq {M_+}/I$.
We allow $I \subset M$ to be the empty set, in which case, we identify $M/I$ with
$M$. If $I = M$, the quotient $M/I$ is identified with the constant 
singleton monoid $\{+\}$. 

A monoid is called  partially cancellative (pc) if there is a cancellative 
monoid $N$ and an ideal $I \subset N$ (possibly empty) such that $M = N/I$.
A monoid $M$ is called partially cancellative torsion-free (pctf)
if there is a cancellative torsion-free
monoid $N$ and an ideal $I \subset N$ (possibly empty) such that $M = N/I$.
Such monoids are called pctf monoids in \cite{CHW-4}.
We shall also use this notation in this paper.

For a monoid $M$ and a ring $R$, we let $R[M]$ denote the free 
$R$-module on $M$. Then $R[M]$ is a commutative $R$-algebra 
in the usual way, with multiplication given by the addition rule for $M$. 
Since $R$ is Noetherian and $M$ is finitely generated by our assumption,
the Hilbert basis theorem tells us that $R[M]$ is also Noetherian.
One also checks that $R[M]$ is an integral domain
if and only if
$R$ is an integral domain and $M$ is cancellative torsion-free
(see \cite[Theorem~4.8]{BG}).
$R[M]$ is reduced if $R$ is reduced and $M$ is cancellative torsion-free.
It can also happen that $R[M]$ is reduced even if $M$ is a torsion monoid. 
For example,
$R[M]$ is reduced if $R$ is reduced and $M$ is cancellative whenever
$\Q \subset R$ (see \cite[Theorem~4.19]{BG}).
If $I$ is an ideal of the monoid $M$, then $R[I]$ is an ideal of the ring
$R[M]$ and $R[M/I] = {R[M]}/{R[I]}$.

Recall that a cancellative monoid $M$ is called normal if
$M = \{a \in \gp(M)| \ na \in M \ \mbox{for \ some} \ n > 0\}$.
One says that $M$ is semi-normal if 
$M = \{a \in \gp(M)| \ 2a, 3a \in M\}$.
The semi-normalization of a cancellative monoid $M$ is
the submonoid $sn(M)$ consisting of all elements $a \in \gp(M)$
such that $2a, 3a \in M$.
Given an inclusion of monoids $M \subset N \subset \gp(M)$ and an
element $a \in N$, we let $M\<a\>$ denote the submonoid of $N$
generated by $M$ and $a$. Given a finite set $F = \{a_1, \cdots , a_r\} 
\subset N$, we can inductively define $M\<F\> \subset N$.
It is then easy to see that
$sn(M) = {\underset{F \subset sn(M)}\colim} M\<F\>$. 
One can also check that $sn(M) = \{a \in \gp(M)| \ na \in M \ 
\mbox{for \ all} \ n \gg 0\}$
(see \cite[\S~1]{Swan}).
We shall let $n(M) = \{a \in \gp(M)| \ na\in M \ \mbox{for \ some} \ n\}$ 
denote the normalization of $M$ in $\gp(M)$.
The following result is elementary.

\begin{lem}\label{lem:ELM-monoid-*}
Let $M$ be a cancellative monoid and let $M \subset M' \subset n(M)$
be inclusions of monoids. Then the following hold.
\begin{enumerate}
\item
$M'$ is cancellative.
\item
$M'$ is torsion-free if $M$ is so.
\item
$M'$ is positive torsion-free if $M$ is so.
\item
$M'$ is positive if $M$ is so provided $M' \subset sn(M)$.
\end{enumerate}
\end{lem}
\begin{proof}
We have the inclusions $M \subset M' \subset n(M) \subset \gp(M)$.
In particular, all these monoids have same group completions.
The parts (1) and (2) of the lemma follow immediately from this.
For (3), suppose that $M$ is positive torsion-free
and let $a \in U(M')$ be a non-zero
element. We then have $a + b = 0$ for some $b \in M'$.
We can choose some $n \gg 0$ such that $na, nb \in M$. Since
$M$ is torsion-free, we must have $na, nb \neq 0$. 
Since $na + nb = n(a+b) = 0$, it follows that $na \in U(M) \setminus \{0\}$.
This contradicts our assumption that $M$ is positive.

We now prove (4). Suppose $M$ is positive and let $a \in U(M')$ be a 
non-zero element. We then have $a + b = 0$ for some $b \in M'$.
Since $M' \subset sn(M)$, we can choose some $n_0 \gg 0$ such that 
$na, nb \in M$ for all
$n \ge n_0$ (see above). If $n_0a \neq 0$, then $n_0b \neq 0$ and
$n_0a + n_0b = 0$, so we get $0 \neq n_0a \in U(M)$, which contradicts our
hypothesis. Suppose that $n_0a = 0$. Then $n_0b = 0$.
But this means that $(n_0+1)a = a \neq 0$ and $(n_0+1)b = b \neq 0$.
On the other hand, we have $(n_0+1)a + (n_0+1)b = (n_0+1)(a+b) = 0$.
Since $(n_0+1)a, (n_0+1)b \in M$, we get $0 \neq (n_0+1)a \in U(M)$.
This again contradicts our hypothesis. We are therefore done.
\end{proof}

\begin{remk}\label{remk:ELM-monoid-*-0}
Note that the part (3) of \lemref{lem:ELM-monoid-*} is not true in general
if $M$ is not torsion-free. In fact, it is easy to see in this case that
$n(M)$ contains the whole torsion subgroup of $\gp(M)$.
\end{remk}

\subsection{Reduction of positive monoids to semi-normal monoids}
\label{sec:Reduction}
In this subsection, we shall establish some reduction steps which will tell us
how to replace a positive cancellative torsion-free monoid in \thmref{thm:MT-1}
by the one which is positive cancellative torsion-free and semi-normal.
Recall that a Cartesian square of commutative rings
\begin{equation}\label{eqn:MilnorSq}
\xymatrix@C1pc{
A_1 \ar@{->}[r]^{\alpha}
     \ar@{->}[d]_{\phi}
&A_2 
     \ar@{->}[d]^{\psi}
\\
B_1 \ar@{->}[r]_{\beta}
     &B_2      
}
\end{equation}
called a {\it Milnor square} if one of $\psi$ and $\beta$
is surjective.
We shall use the following consequence of the classical results of Bass 
\cite{Bass} and Milnor \cite{Milnor} as one important tool.

\begin{prop}\label{prop:Milnor*}
Given a Milnor square ~\eqref{eqn:MilnorSq}, 
the map $K_i(A_1, B_1) \to K_i(A_2, B_2)$ 
of relative $K$-groups is an isomorphism for $i \le 0$. In 
particular, there is a long exact sequence sequence of algebraic $K$-groups
\[
K_i(A_1) \rightarrow K_i(A_2)\oplus K_i(B_1) \rightarrow K_i(B_2) \rightarrow 
K_{i-1}(A_1) \rightarrow \cdots 
\]
for $i \le 1$.
 
If $\psi$ and $\beta$ are both surjective in the Milnor square, then 
$K_i(A_1, B_1) \to K_i(A_2, B_2)$ is an isomorphism for $i \le 1$.
In particular, the above sequence is exact for all $i\leq2$.
\end{prop}

\begin{lem}\label{lem:affred}
Let $A$ be a ring and $I$ a nilpotent ideal of $A$. 
Then $K_i(A) = K_i(A/I)$ for all $i\leq 0$.
\end{lem}
\begin{proof}
See Bass' book \cite[Chapter IX, Proposition 1.3]{Bass}.
\end{proof}

\begin{lem}\label{lem:RemoveSemi}
Let $R$ be a ring and $M$ a positive cancellative 
torsion-free monoid.
Then $K_i(R[M])\simeq K_i(R[sn(M)])$ for all $i\leq -1$.
\end{lem}
\begin{proof}
We can assume $R$ to be reduced by \lemref{lem:affred}.
It follows from \cite[Theorem~4.19]{BG} that $R[M]$ is also reduced.
As in \S~\ref{sec:Monoids}, we can write 
$sn(M) = {\underset{F \subset sn(M)}\colim} M\<F\>$,
where $F$ is a finite set of elements $a \in \gp(M)$ such that
$2a, 3a \in M$. Since $K$-theory commutes with direct limits, 
it is enough to prove that $K_i(R[M]) \simeq K_i(R[M\<F\>])$
for $i \le -1$. Since $M\<F\>$ is an iterated extension of
monoids of the form $N\<\{a\}\>$ such that $a \in \gp(N)$ with $2a, 3a \in N$,
it will suffice to prove inductively that
$K_i(R[M]) \simeq K_i(R[M\<\{a\}\>])$ for $i \le -1$, where
$a \in \gp(N)$ with $2a, 3a \in M$.
Note here that $M\<\{a\}\>$ is a positive cancellative 
torsion-free monoid if $M$ is so and $a \in sn(M)$ by \lemref{lem:ELM-monoid-*}.

We let $A = R[M]$ and $B = R[M\<\{a\}\>]$. Let
$F(A)$ denote the total ring of quotients of $A$. Note that $F(A)$ is a 
product of fields. We can 
write $B = A[x] \subset F(A)$, where $x^2, x^3 \in A$.
We then get a conductor Milnor square
\begin{equation}\label{eqn:SN-0}
\xymatrix@C1pc{
A \ar@{->}[r] \ar@{->}[d] & B \ar@{->}[d] \\
A/I \ar@{->}[r] & B/I}
\end{equation}
with conductor ideal $I =(x^2,x^3)$. 

Furthermore, one knows in this case that
$(A/I)_{\rm red} \to (B/I)_{\rm red}$ is an isomorphism 
(see, for instance, \cite[Proof of Theorem~14.1]{Swan}).
This gives a commutative diagram of long exact sequences 
{\small
    \[
     \xymatrix{
     K_{i+1}(A/I) \ar@{->}[r]\ar@{->}[d]& K_i(A, I) \ar@{->}[r]\ar@{->}[d] &
K_i(A)
     \ar@{->}[r]\ar@{->}[d]& K_i(A/I) \ar@{->}[r]\ar@{->}[d]& K_{i-1}(A,I)
\ar@{->}[d]\\
K_{i+1}(B/I) \ar@{->}[r] & K_i(B,I) \ar@{->}[r]
     &K_i(B)\ar@{->}[r]&  K_i(B/I) \ar@{->}[r] & K_{i-1}(B,I).}
\]
}

Since $i \le -1$ and $(A/I)_{\rm red} \xrightarrow{\cong} (B/I)_{\rm red}$, it
follows from \lemref{lem:affred} that the 
first and the fourth (from the left) vertical arrows are isomorphisms.
The second and the fifth vertical arrows are isomorphisms
by \propref{prop:Milnor*}. We conclude that the middle vertical arrow
is also an isomorphism.
\end{proof}

\subsection{Reduction to positive monoids}\label{sec:Reduction-00}
In this subsection, our goal is to prove a reduction step which will allow
us to replace a cancellative torsion-free monoid in \thmref{thm:MT-1}
by the one which is cancellative torsion-free and positive.

\begin{lem}\label{lem:Laurent-K_{-d}}
Let $A$ be a ring of dimension $d \ge 0$. Then the canonical map
$K_{i}(A)\to K_{i}(A[X_1,\ldots,X_n,Y_1^{\pm 1},\ldots,Y_m^{\pm 1}])$ is an
isomorphism for all $i\leq -d$ and $n,m\geq 0$.
\end{lem}
\begin{proof}
The proof is by induction on $m$. The case $m=0$ is due to 
\cite[Theorem~B(ii)]{KST16}. Let $m \ge 1$ and 
assume that $K_{i}(A)\to 
K_{i}(A[X_1,\ldots,X_n,Y_1^{\pm 1},\ldots,Y_{m-1}^{\pm 1}])$ is an isomorphism
for all $n \ge 0$.

We let $B=A[X_1,\ldots,X_n,Y_1^{\pm 1},\ldots,Y_{m-1}^{\pm 1}]$. 
The induction hypothesis implies that
$K_i(A)\cong K_i(B)$ for all $i\leq -d$ and all $n \ge 0$. 
Hence by \cite[Theorem~B(i)]{KST16}, we get
$K_{i-1}(B)=0$ for all $i\leq -d$.
 
We now consider the commutative diagram of $K$-groups
 
 \[
  \xymatrix
{
K_i(B)\oplus K_i(B)
    \ar@{->}[d]^{\xi_i}
\ar@{->>}[r]^-{\beta_i}
& K_i(B)
    \ar@{->}[d]^{\alpha_i} \\
 K_i(B[Y_m]) \oplus K_i(B[Y_m^{-1}])
\ar@{->}[r]
& K_i(B[Y^{\pm 1}_m])
\ar@{->}[r]
&K_{i-1}(B)=0,
 }
\]
where the bottom row is the Bass fundamental exact sequence.
The map $\beta_i$ is defined by $\beta_i(a, b) = a-b$. This is
clearly surjective.
By induction, we have 
\[
K_i(B) \cong K_i(A) \cong 
K_i(A[X_1,\ldots, X_n, Y_m, Y_1^{\pm 1},\ldots,Y_{m-1}^{\pm 1}])
= K_i(B[Y_m]),
\]
where the second isomorphism holds by replacing $n$ by $n+1$.
By the same reason, we get 
$K_i(B)\cong K_i(A)\cong K_i(B[Y_m^{-1}])$ for all $i\leq -d$.
Hence, $\xi_i$ is an isomorphism for all $i\leq -d$.
A diagram chase shows that $\alpha_i$ is surjective.
Since this is already split injective, we are done. 
\end{proof}

\begin{lem}\label{lem:RemovePosi}
Let $M$ be a monoid and $N = M \setminus U(M)$.
Let $R$ be a ring of dimension $d \ge 0$.
If $K_i(R)\xrightarrow{\simeq} K_i(R[N_*])$ for all $i\leq -d$, then 
$K_i(R) \xrightarrow{\simeq} K_i(R[M])$ for all $i\leq -d$.
\end{lem}
\begin{proof}
By \cite[Lemma~6.1]{Gub-7} (the proof thereof), there is a Milnor square
\begin{equation}\label{eqn:Unit}
\xymatrix@C1pc{
R[N_*] \ar@{->}[r] \ar@{->}[d]& R[M] \ar@{->}[d]
\\
R \ar@{->}[r]
     &R[U(M)],     
}
\end{equation}
where $\phi$ is split surjective because the composite
$R[U(M)] \to R[M] \to R[U(M)]$ is identity (see also \cite[\S~15]{Swan}).
%Note that $N_*$ is semi-normal if $M$ is so, as one easily checks.

Since ~\eqref{eqn:Unit} is a commutative diagram of $R$-algebras, the trivial
Milnor square 
\begin{equation}\label{eqn:Unit-0}
\xymatrix@C1pc{
R \ar@{->}[r] \ar@{->}[d]& R \ar@{->}[d]
\\
R \ar@{->}[r]
     &R
}
\end{equation}
(where all maps are identity)
maps to this square.
It follows from \propref{prop:Milnor*} that 
there is a commutative diagram of short exact sequences
\begin{equation}\label{eqn:Unit-1}
\xymatrix@C.8pc{
0 \ar@{->}[r]& K_i(R) \ar@{->}[r]\ar@{->}[d] &K_i(R)\oplus K_i(R) 
     \ar@{->}[r]\ar@{->}[d]& K_i(R) \ar@{->}[r]\ar@{->}[d]& 0\\
     0 \ar@{->}[r] & K_i(R[N_*]) \ar@{->}[r]
     &K_i(R)\oplus K_i(R[M]) 
\ar@{->}[r]&  K_i(R[U(M)]) \ar@{->}[r] & 0.
}
\end{equation}

The left vertical arrow is an isomorphism by our assumption and the
right vertical arrow is an isomorphism by \lemref{lem:Laurent-K_{-d}}
since $U(M)$ is torsion-free (and hence free).
This implies the desired assertion. 
\end{proof}

\subsection{Cancellative to partially 
cancellative monoids}\label{sec:Reduction-*}
We shall now show how we can reduce the proof of 
\thmref{thm:MT-1} to the case of cancellative torsion-free monoids.
The result that we shall use is the following. 

\begin{lem}\label{lem:AffMainVans}
Let $M$ be a partially cancellative torsion-free monoid and 
$R$ a ring of dimension $d \ge 0$.
Assume that $K_i(R) \to K_i(R[N])$ is an isomorphism for all $i \le -d$ 
and for all cancellative torsion-free monoids $N$.
Then $K_i(R) \to K_i(R[M])$ is an isomorphism for all $i \le -d$.
\end{lem}
\begin{proof}
By definition of partially cancellative torsion-free monoids, we can write
$M = N/I$, where $N$ is a cancellative torsion-free monoid.
This yields a Milnor square
\begin{equation}\label{pc-Milnor}
\xymatrix{
R[I_*] \ar@{->}[r]
     \ar@{->}[d]
&R[N]
     \ar@{->}[d]
\\
R \ar@{->}[r]
     &R[N/I].     
}
\end{equation}

Mimicking the proof of \lemref{lem:RemovePosi}, we get a commutative diagram
of long exact sequences
\begin{equation}\label{eqn:Unit-2}
\xymatrix@C.6pc{
K_i(R) \ar@{->}[r]\ar@{->}[d] & K_i(R)\oplus K_i(R) 
     \ar@{->}[r]\ar@{->}[d]& K_i(R) \ar@{->}[r]\ar@{->}[d] & K_{i-1}(R) \ar[d] 
\ar[r] & K_{i-1}(R)\oplus K_{i-1}(R) \ar[d] \\
K_i(R[I_*]) \ar@{->}[r] & K_i(R)\oplus K_i(R[N]) 
\ar@{->}[r]&  K_i(R[N/I]) \ar@{->}[r] & K_{i-1}(R[I_*]) \ar[r] &
K_{i-1}(R)\oplus K_{i-1}(R[N]).
}
\end{equation}

All vertical arrows in this diagram except possibly the middle one are
isomorphisms because $N$ and $I_*$ are cancellative torsion-free 
monoids. It follows that the middle vertical arrow is an isomorphism too.
This finishes the proof.
\end{proof}

An identical argument also proves the following variant of
\lemref{lem:AffMainVans}. We shall not use this here but it may be useful
in answering Question~\ref{question-*}. 

\begin{lem}\label{lem:AffMainVans-torsion}
Let $M$ be a partially cancellative monoid and 
$R$ a ring of dimension $d \ge 0$.
Assume that $K_i(R) \to K_i(R[N])$ is an isomorphism for all $i \le -d$ 
and for all cancellative monoids $N$.
Then $K_i(R) \to K_i(R[M])$ is an isomorphism for all $i \le -d$.
\end{lem}

We shall also use the following result in the proof of \thmref{thm:MT-1}.

\begin{lem}\label{lem:reg-neg}
Let $R$ be a regular ring and $M$ a partially cancellative torsion-free monoid.
Then $K_i(R[M])=0$ for all $i\leq -1$.
\end{lem}
\begin{proof}
If $M$ is cancellative torsion-free, the lemma follows from 
\cite[Theorem~1.3]{Gub-4}. In this case, we can also prove it
using Lemma~\ref{lem:RemoveSemi}. This lemma allows us to assume
that $M$ is semi-normal. Now, by the Gubeladze--Swan theorem 
(see \cite[Corollary~1.4]{Swan}), we have $K_0(R)\cong K_0(R[M \oplus \Z])$.
Together with the fundamental exact sequence, this implies that
$K_i(R[M])=0$ for all $i\leq -1$. 

In case $M$ is partially cancellative torsion-free, 
we can write  $M=N/I$, where $N$ is 
cancellative torsion-free. Hence, by similar arguments as in 
Lemma~\ref{lem:AffMainVans}, we get a Milnor square
 
 \begin{equation}\label{pc-Milnor-1}
 \xymatrix{
 	R[I_*] \ar@{->}[r]
 	\ar@{->}[d]
 	&R[N]
 	\ar@{->}[d]^{\phi}
 	\\
 	R \ar@{->}[r]
 	&R[N/I].     
 }
 \end{equation}

Using \propref{prop:Milnor*}, this yields the long exact sequence 
 \begin{equation}\xymatrix{
 \cdots \ar@{->}[r] & K_i(R)\oplus K_i(R[N]) 
 \ar@{->}[r]&  K_i(R[N/I]) \ar@{->}[r] & K_{i-1}(R[I_*]) \ar[r] &
 \cdots
}
\end{equation}
for $i \le -1$. 
Note that $I_*$ is a cancellative torsion-free monoid. 
Hence, we have $K_i(R[I_*])=0$ for all $i\leq -1$. Since
$K_i(R[N])$ and $K_i(R)$ are both zero for $i\leq -1$, we get 
$K_i(R[N/I])=0$ for $i\leq -1$. This finishes the proof.
\end{proof}

\section{The monoid version of Weibel's conjecture}\label{sec:MR}
The goal of this section is to prove \thmref{thm:MT-1}.
We fix a field $k$. 
%We let $\eft_k$ denote the category of essentially of finite type $k$-algebras and let $\rft_k$ denote the full subcategory of $\eft_k$ consisting of regular $k$-algebras. 
We let $\Sch_k$ denote the category of
separated schemes which are essentially of finite type over $k$ and let 
$\Sm_k$ denote the
full subcategory of $\Sch_k$ consisting of schemes which are regular. 
We shall denote the product of $X, Y \in \Sch_k$ over $k$ by $X \times Y$.
Let $\wt{\Sch}_k$ denote the category of separated Noetherian schemes over $k$.
We let ${\wt{\Sch_k}}/{\zar}$ denote the Grothendieck site on 
${\wt{\Sch}}_k$ given by the Zariski topology.
We shall consider all cohomology groups with respect to the Zariski topology 
in this paper.

\subsection{Quasi-excellent schemes and resolution of singularities}
\label{sec:QE}
Recall from \cite[Chapter~32, p.~260]{Matsumura} that a (Noetherian) ring $A$
is called excellent if the following hold.
\begin{enumerate}
\item
The fibers of the map $A_{\fm} \to \wh{A_{\fm}}$ are geometrically regular
for every maximal ideal $\fm \in A$.
\item
The regular locus $X_{\rm reg}$ of every finite type affine scheme $X$ over $A$
is open in $X$. 
\item
$A$ is universally catenary.
\end{enumerate}

If $A$ satisfies only (1) and (2), it is called quasi-excellent.
A Noetherian separated scheme $X$ is called (quasi-) excellent if is it
covered by the spectra of (quasi-) excellent rings.
One important property of (quasi-) excellent schemes we shall use in this
paper is that if $X$ is a (quasi-) excellent scheme and $X' \to X$
is an essentially of finite type morphism, then $X'$ is 
also (quasi-) excellent. In particular, all objects of $\Sch_k$ are 
excellent. This is also true if we replace $k$ by any complete local ring.

%Let $\sC_k$ be a subcategory of ${\wt{\Sch}}_k$ which is closed under
%finite type morphisms. For example, $\sC_k$ could be $\Sch_k$ or
%the category of excellent or the category of quasi-excellent schemes in
%${\wt{\Sch}}_k$.

Let $\sC_k$ be a subcategory of ${\wt{\Sch}}_k$.
We shall say that $\sC_k$ admits weak resolution of singularities if the
following hold.
\begin{enumerate}
\item
If $X \in \sC_k$ and $Y \to X$ is a finite type morphism in ${\wt{\Sch}}_k$,
then $Y \in \sC_k$.
\item 
Given a reduced scheme $X \in \sC_k$, there exists a Cartesian square of 
schemes 
\begin{equation}\label{eqn:CarSq}
\xymatrix{
\widetilde{Y} \ar@{->}[r]
     \ar@{->}[d]
&\widetilde{X}
     \ar@{->}[d]
\\
Y \ar@{->}[r]
     &X      
}
\end{equation}
such that 
$\wt X \rightarrow X$ is a proper, the
horizontal arrows are nowhere dense closed immersions, $\wt{X} \in \Sm_k$
and $\wt X \setminus \wt Y \rightarrow X\setminus Y$
is an isomorphism. 
Note that $\wt{X}, \ Y$ and $\wt{Y}$ belong to $\sC_k$ by (1). 
\end{enumerate}

We shall use the following celebrated result of Hironaka 
\cite[Theorem 1*]{Hiro} and its extension to quasi-excellent schemes by Temkin 
\cite[Theorem~1.1]{Temkin}.

\begin{thm}\label{thm:Res-sing}
The category of quasi-excellent schemes over $\Q$ admits weak resolution of 
singularities.
\end{thm}

%Until \S~\ref{sec:Positive}, we fix a field $k$ which admits weak resolution 
%of singularities. 

\subsection{The spectrum $C^MK$}\label{sec:C-M}
We let $X \mapsto K(X)$ denote the non-connective Thomason-Trobaugh algebraic 
$K$-theory presheaf of spectra on $\wt{\Sch}_k$
(denoted by $K^B(X)$ in \cite{TT}). Given a monoid $M$, we have an
augmented $k$-algebra $k[M]$. For any $X \in \wt{\Sch}_k$, we let
$X[M] = X \times_k \Spec(k[M]) \cong X \times_{\Spec(\Z)} \Spec(\Z[M])$. 
The augmentation of $k[M]$ yields
natural maps $X \xrightarrow{\sigma_X} X[M] \xrightarrow{p_X} X$
whose composite is identity. In particular, there is functorial
decomposition $K(X[M]) \simeq K(X) \amalg C^MK(X)$, where
$C^MK(X) = {\rm hofiber}(K(X[M]) \xrightarrow{\sigma^*_X} K(X))$.
It follows that $X \mapsto C^MK(X)$ is a presheaf of spectra on 
${\wt{\Sch}}_k$. 
For any $i \in \Z$, we shall let $C^MK_i(X)$ denote the stable homotopy group
of $\pi_i(C^MK(X))$.  
We let $C^M\sK_{i,X}$ denote the Zariski sheaf on $X$ associated to the 
presheaf $U \mapsto \pi_i(C^MK(U))$. 
 
If $f:Y \to X$ is a morphism in ${\wt{\Sch}}_k$, we shall let
$C^MK(X,Y)$ denote the homotopy fiber of the map
$f^*: C^MK(X) \to C^MK(Y)$. We let $C^MK_i(X,Y)$ denote the homotopy groups
of $C^MK(X,Y)$. We let $C^M\sK_{i,(X,Y)}$ denote the Zariski sheaf on 
$X$ associated to the presheaf $U \mapsto C^MK_i(U, U \times_X Y)$.

\begin{thm}$($\cite[Theorem~10.3]{TT}$)$\label{thm:TT-gen}
Given $X \in {\wt{\Sch}}_k$ and a closed subscheme $Z \subset X$, there exists 
a strongly convergent spectral sequence
\begin{equation}\label{eqn:TT-S}
E^{p,q}_2 = H^p(X, C^M\sK_{q,(X, Z)}) \Rightarrow C^MK_{q-p}(X,Z).
\end{equation}
\end{thm}
\begin{proof}
This is proved by repeating the argument of \cite[Theorem~10.3]{TT}
verbatim with the aid of \cite[Proposition~3.20.2]{TT} and the fact that
if $U = \Spec(A)$ is an affine open in $X$ with a prime ideal $\fp \subset A$
such that $A_{\fp} = {\underset{i}\varinjlim} \ A[f^{-1}_i]$,
then $A_{\fp} \otimes_k k[M] \simeq {\underset{i}\varinjlim} \ 
(A[M])[f^{-1}_i]$. We leave out the details.
\end{proof} 

\begin{lem}\label{lem:waytored}
Let $X \in \wt{\Sch}_k$ be of dimension $d$ and $X_{\red}$ denote the 
reduced subscheme of $X$. 
Let $M$ be any monoid. Then $C^MK_i(X, X_{\rm red}) = 0$ for 
$i < -d$. If $M$ is cancellative torsion-free, then
$C^MK_i(X, X_{\rm red}) = 0$ for $i \le -d$. In particular,
$C^MK_i(X) \simeq C^MK_i(X_{\red})$ for all $i\leq -d$.
\end{lem}
\begin{proof}
The first assertion is an immediate consequence of
the spectral sequence ~\eqref{eqn:TT-S} and \lemref{lem:affred}.
We show the second part.

Using the spectral sequence ~\eqref{eqn:TT-S}, it suffices to show that
$C^MK_i(A, I) = 0$ for $i \le 0$ if $A$ is the local ring of a Zariski
point of $X$ and $I$ is its nil-radical.
We shall prove the stronger result that
$K_i(A[M], I[M]) = 0$ for $i \le 0$. 
Using \lemref{lem:affred} and the long exact sequence for relative $K$-theory,
we only need to show that the map $K_1(A[M]) \to K_1(A/I[M])$ is 
surjective.

We now consider the commutative diagram of short exact sequences
\begin{equation}\label{eqn:waytored-0}
\xymatrix@C.8pc{
0 \ar[r] & SK_1(A[M]) \ar[r] \ar[d] & K_1(A[M]) \ar[r] \ar[d] &
(A[M])^{\times} \ar[r] \ar[d] & 0 \\
0 \ar[r] & SK_1(A/I[M]) \ar[r] & K_1(A/I[M]) \ar[r] &
(A/I[M])^{\times} \ar[r] & 0.}
\end{equation}

It follows from our assumption on $M$ and
\cite[Theorem~4.19]{BG} that $A/I[M] = (A[M])_{\rm red}$. We conclude from
\cite[Chap.~IX, Propositions~3.10 and 3.11]{Bass} 
(see also \cite[Chapter~III, Lemma~2.4]{Weibel-2}) that
the left vertical arrow in ~\eqref{eqn:waytored-0} is an isomorphism.

To prove that the right vertical arrow in ~\eqref{eqn:waytored-0} 
is surjective (which will finish the proof), we consider the commutative
diagram
\begin{equation}\label{eqn:waytored-1}
\xymatrix@C.8pc{
A^{\times}  \times U(M) \ar[r] \ar[d] & (A[M])^{\times} \ar[d] \\
(A/I)^{\times} \times U(M) \ar[r]^-{\simeq} & (A/I[M])^{\times}.}
\end{equation}

Since $A/I$ and $A/I[M]$ are reduced (observed above) and $A$ is local,
it follows from \cite[Proposition~4.20]{BG} that the lower horizontal
arrow is an isomorphism. Since the left vertical arrow is clearly
surjective (uses again that $A$ is local), we conclude that
the right vertical arrow must also be surjective, as desired. 
\end{proof}

\subsection{The main result}\label{sec:Main-prf}
We are now ready to prove \thmref{thm:MT-1}, which extends the assertion of
Weibel's $K$-dimension conjecture from polynomial to monoid algebras.
The following result is a refinement of \cite[Proposition~6.1]{KST16}.
We do not use it here but include it because it may be
useful in the generalization of \thmref{thm:MT-1} for non-affine schemes. 
We fix a field $k$.

\begin{lem}\label{lem:KS-refine}
Let $M$ be a monoid and $X \in {\wt{\Sch}}_k$.
Assume that $C^MK_i(\sO_{X,x}) = 0$ for all $i \le -\dim(\sO_{X,x})$
and all points $x \in X$. Then $C^MK_i(X) = 0$ for all $i \le -\dim(X)$.
\end{lem}
\begin{proof}
For $i < - \dim(X)$, this is stated in \cite[Proposition~6.1]{KST16} and 
proven in \cite[Proposition~3]{KS16}.  However, the proof given there 
works in the modified case as well with no change. We briefly explain it.

Let $d$ denote the dimension of $X$.
Using the spectral sequence ~\eqref{eqn:TT-S}, it suffices to show that
$H^p(X, C^M\sK_{q,X}) = 0$ whenever $q-p \le -d$.
Suppose first that $q+d \le 0$. In this case, it suffices to show that
$C^M\sK_{q,X} = 0$. But this follows from our assumption because 
$\dim(\sO_{X,x}) \le d \le -q$ for all $x \in X$.

We now fix $0 \le p \le d$ and $q-p \le -d$ 
(equivalently, $p > q+d-1$) such that $q+d > 0$.
By \cite[Lemma~4]{KS16} (where we take $r = q+d-1$), it suffices to show that
if $x\in X$ is a Zariski point with $\dim(\ov{\{x\}}) \ge q+d$, then
$C^MK_q(\sO_{X,x}) = 0$. 
But the condition $\dim(\ov{\{x\}}) \ge q+d$ implies that
$\dim(\sO_{X,x}) \le -q$. Our hypothesis implies again that
$C^MK_q(\sO_{X,x}) = 0$. This finishes the proof.
\end{proof}

\begin{lem}\label{lem:Regular}
Let $M$ be a cancellative torsion-free semi-normal monoid and 
$X$ a regular Noetherian scheme of dimension $d \ge 0$. Then
$C^MK_i(X) = 0$ for $i \le -d$.
\end{lem}
\begin{proof}
Since $X$ is regular, we can assume that it is connected. We first assume that 
$X$ is affine. In this case, it follows from \cite[Corollary~1.4]{Swan}
that $C^MK_0(X) = 0$. The same holds if we replace $X$ by $X[T^{\pm 1}]$ because
the latter scheme is also affine and regular.

On the other hand, the fundamental exact sequence of Bass yields a 
surjective homomorphism $C^MK_i(X[T^{\pm 1}]) \surj C^MK_{i-1}(X)$.
We conclude inductively that $C^MK_i(X) = 0$ for $i \le 0$.  
Using the spectral sequence ~\eqref{eqn:TT-S}
(see the proof of \lemref{lem:waytored}), we now conclude that
$C^MK_i(X) = 0$ for $i \le -d$ if $X$ is any regular Noetherian 
scheme of dimension $d$.
\end{proof}

\begin{lem}\label{lem:scheme-main}
Let $\sC_k$ be a subcategory of ${\wt{\Sch}}_k$ which admits weak resolution of 
singularities. Let $X \in \sC_k$ be of dimension $d \ge 0$. 
Let $M$ be a cancellative torsion-free
semi-normal monoid. Then $K_{i}(X) \to K_{i}(X[M])$ is an
isomorphism for all $i \le -d$.
\end{lem}
\begin{proof}
The lemma is equivalent to showing that $C^MK_i(X) = 0$ for $i \le -d$.
By \lemref{lem:waytored}, we can assume that $X$ is reduced.
Consequently, we can assume that $X[M]$ is reduced by 
\cite[Theorem~4.19]{BG} as $M$ is cancellative torsion-free.
Note that $X[M] \in \sC_k$ as $M$ is finitely generated.
We shall prove the lemma by induction on $d$. The $d = 0$ case follows
from \lemref{lem:Regular}. So we assume that $d \ge 1$.

%Suppose first that $X$ is of finite type over $k$.
Due to the assumption of weak resolution of singularity, we have an
abstract blow-up square (see \cite[Introduction]{CHSW} for definition)
as in ~\eqref{eqn:CarSq}.
Adjoining our monoid, this yields another abstract blow-up square in $\sC_k$:
\begin{equation}\label{eqn:CarSq-1}
\xymatrix@C1pc{
\widetilde{Y}[M] \ar@{->}[r]
     \ar@{->}[d]
&\widetilde{X}[M]
     \ar@{->}[d]
\\
Y[M] \ar@{->}[r]
     & X[M].   
}
\end{equation}

For any integer $n \ge 1$, we let $nY$ denote the infinitesimal thickening of 
$Y$ inside $X$ defined by the sheaf of ideals $\sI^n_Y$, where 
$\sI_Y$ is the sheaf of ideals on $X$ defining $Y$. We define $n\wt{Y}$
analogously. Then it is easy to see that 
$nY[M]$ (resp. $n\wt{Y}[M]$)
is an infinitesimal thickening of $Y[M]$ (resp. $\wt{Y}[M]$)
inside $X[M]$ (resp. $\wt{X}[M]$).

Since $C^MK_*(X)$ is a natural direct factor
of $K_*(X[M])$, we can apply the pro-cdh-descent theorem 
\cite[Theorem A]{KST16} to get an exact sequence of pro-abelian groups
\begin{equation}\label{eqn:CarSq-2}
\{C^MK_{i+1}(n\wt Y)\} \to C^MK_{i}(X) \to 
 \{C^MK_{i}(nY)\} \oplus C^MK_{i}(\wt X).
\end{equation}

Since $Y \subset X$ and $\wt{Y} \subset \wt{X}$ are nowhere dense closed
subsets, we see that $\dim(Y)$ and $\dim(\wt{Y})$ are less than $d$.
Using induction on $d$ and regularity of $\wt{X}$, we see that the
end terms of this exact sequence vanish. We conclude that
$C^MK_{i}(X) = 0$ for $i \le -d$. 
\end{proof}

Combining \thmref{thm:Res-sing} and \lemref{lem:scheme-main}, we get:

\begin{cor}\label{cor:Quais-exc-0}
Let $M$ be a cancellative torsion-free semi-normal monoid and $X$ a 
quasi-excellent
scheme over $\Q$ of dimension $d \ge 0$. Then $K_{i}(X) \to K_{i}(X[M])$ is an
isomorphism for all $i \le -d$.
\end{cor}

We now state our main result on the extension of Weibel's conjecture
to monoid algebras.

\begin{thm}\label{thm:MT-1-*}
Let $M$ be a partially cancellative torsion-free monoid and let
$\sC_k$ be a subcategory of ${\wt{\Sch}}_k$ which admits weak resolution of
singularities. Let $X \in \sC_k$ be affine of dimension $d \ge 0$.
Assume that one of the following holds.
\begin{enumerate}
\item
$d \ge 1$.
\item
$M$ is cancellative and semi-normal.
\end{enumerate}
Then the map $K_{i}(X) \to K_{i}(X[M])$ is an isomorphism for all $i \le -d$.
\end{thm}
\begin{proof}
In view of the main results of \cite{KST16}, 
the theorem is equivalent to proving that $C^MK_i(X) = 0$ for $i \le -d$.
Lemma~\ref{lem:waytored} allows us to assume that $X$ is reduced.
By \lemref{lem:AffMainVans},
we can assume that $M$ is a cancellative torsion-free monoid.  
Using \lemref{lem:RemovePosi}, we can further assume that $M$ is positive.
If $d \ge 1$, we can use \lemref{lem:RemoveSemi} to assume that 
$M$ is a cancellative torsion-free semi-normal monoid.
If $d = 0$, then $M$ is already given to be cancellative torsion-free 
semi-normal.

We have therefore reduced the proof of the theorem to the case
where $M$ is a cancellative torsion-free semi-normal monoid
and $X \in \sC_k$ a reduced affine scheme of dimension $d \ge 0$.
We can therefore apply \lemref{lem:scheme-main}
to conclude the proof.
\end{proof}

Combining \corref{cor:Quais-exc-0} and \thmref{thm:MT-1-*}, we get:

\begin{cor}\label{cor:Quais-exc-00}
Let $M$ be a partially cancellative torsion-free monoid and let
$X$ be a quasi-excellent affine scheme of dimension $d \ge 0$ over $\Q$.
Assume that one of the following holds.
\begin{enumerate}
\item
$d \ge 1$.
\item
$M$ is cancellative and semi-normal.
\end{enumerate}
Then the map $K_{i}(X) \to K_{i}(X[M])$ is an isomorphism for all $i \le -d$.
\end{cor}

\vskip .3cm

\begin{remk}\label{remk:Dim-3}
One can easily check from the proof of \lemref{lem:scheme-main}
that the above 
proof of \thmref{thm:MT-1-*} remains valid (without any change) for all
$d$-dimensional schemes over any ground field
$k$ which admits weak resolution of singularities for schemes of
dimensions up to $d$. Since the resolution of singularities is known
to hold in dimension up to three over any ground field (see \cite{VO}),
we see that the assertion of \thmref{thm:MT-1-*} remains valid for
affine schemes over any arbitrary field as long as $d \le 3$.

We also remark that by the same reason as above, \thmref{thm:MT-1-*}
is also valid if $k$ is any field
and $X$ is either an affine normal crossing scheme or an affine toric variety 
over $k$. This is because $X$ admits resolution of singularities in both cases. 
\end{remk}

\subsection{Vanishing of $K_{< -d}(X[M])$}\label{sec:Vanish*}
In this subsection, we shall prove a vanishing result for
$K_{< -d}(X[M])$ if $X$ is a quasi-excellent $\Q$-scheme of dimension $d$ and $M$
is a cancellative torsion-free monoid. We need the following extension 
of \lemref{lem:Regular}.

\begin{lem}\label{lem:Regular-vanish}
Let $M$ be a cancellative torsion-free monoid and 
$X$ a regular Noetherian scheme of dimension $d \ge 0$. Then
$K_i(X[M]) = 0$ for $i < -d$.
\end{lem}
\begin{proof}
Since $X$ is regular, we can assume that it is connected. We first assume that 
$X$ is affine. In this case, it follows from \cite[Theorem~1.3]{Gub-4}
that $K_{i}(X[M]) = 0$ for $i \le -1$. 
Suppose now that $X$ is any Noetherian scheme of dimension $d$.
Using \cite[Theorem~B]{KST16}, it suffices to show that
$C^MK_i(X) = 0$ for $i < -d$. But this follows immediately 
from the spectral sequence ~\eqref{eqn:TT-S}
(see the proof of \lemref{lem:waytored}).
\end{proof}

\begin{thm}\label{thm:scheme-main-vanish}
Let $M$ be a cancellative torsion-free monoid and $X$ a quasi-excellent 
$\Q$-scheme of dimension $d$. Then $K_i(X[M]) = 0$ for $i < -d$.
\end{thm}
\begin{proof}
This proof is identical to that of \lemref{lem:scheme-main}.
We give the sketch. 
We prove the theorem by induction on $d$. The $d = 0$ case follows
from \lemref{lem:Regular-vanish}. So we assume that $d \ge 1$.

By \lemref{lem:waytored}, we can assume that $X$ is reduced.
Consequently, we can assume that $X[M]$ is reduced by 
\cite[Theorem~4.19]{BG} as $X$ is a $\Q$-scheme.

By \thmref{thm:Res-sing}, we have an
abstract blow-up square as in ~\eqref{eqn:CarSq}.
Adjoining our monoid, this yields another abstract blow-up square:
\begin{equation}\label{eqn:CarSq-1**}
\xymatrix@C1pc{
\widetilde{Y}[M] \ar@{->}[r]
     \ar@{->}[d]
&\widetilde{X}[M]
     \ar@{->}[d]
\\
Y[M] \ar@{->}[r]
     & X[M].   
}
\end{equation}

Now, as we did in the proof of \lemref{lem:scheme-main},
we can apply the pro-cdh-descent theorem 
\cite[Theorem A]{KST16} to get an exact sequence of pro-abelian groups
\begin{equation}\label{eqn:CarSq-2**}
\{K_{i+1}(n\wt Y[M])\} \to K_{i}(X[M]) \to 
 \{K_{i}(nY[M])\} \oplus K_{i}(\wt X[M]).
\end{equation}

Since $Y \subset X$ and $\wt{Y} \subset \wt{X}$ are nowhere dense closed
subsets, we see that $\dim(Y)$ and $\dim(\wt{Y})$ are less than $d$.
Using induction on $d$ and regularity of $\wt{X}$, we see that the
end terms of this exact sequence vanish. We conclude that
$K_{i}(X[M]) = 0$ for $i < -d$. 
\end{proof}

\subsection{The positive characteristic case}\label{sec:Inv-p}
The weak resolution of singularities in all dimensions is yet unknown
if the ground field $k$ has positive characteristic.
Nevertheless, we can show that \thmref{thm:MT-1-*} is valid
in this case too if we invert 
${\rm char}(k)$. Using the reduction steps of \S~\ref{sec:prelim},
this turns out to be actually an easy consequence of 
the homotopy invariance property of Weibel's homotopy $K$-theory. 
More precisely, we can prove the following.
 
\begin{thm}\label{thm:MT-1-p}
Let $k$ be a field of characteristic $p > 0$. Let
$X \in {\wt{\Sch}}_k$ be an affine scheme of dimension $d \ge 0$.
Assume that one of the following holds.
\begin{enumerate}
\item
$d \ge 1$.
\item
$M$ is cancellative and semi-normal.
\end{enumerate}
Then the map $K_{i}(X)[\tfrac{1}{p}] \to K_{i}(X[M])[\tfrac{1}{p}]$ is an 
isomorphism for all $i \le -d$.
\end{thm}
\begin{proof}
In this proof, we shall work only with $\Z[\tfrac{1}{p}]$-modules.
In particular, $K_i(X)$ will mean $K_i(X)[\tfrac{1}{p}]$ for simplicity of
notation. In this case, we can replace the spectrum $K(X)$ by $KH(X)$
(see \cite[Exercise 9.11(h)]{TT}), where the latter is Weibel's
homotopy $K$-theory \cite{Weibel89}.

Now, we can assume $M$ to be cancellative torsion-free and positive
by using the $KH$-analogues of Lemmas~\ref{lem:AffMainVans} and 
~\ref{lem:RemovePosi}.
If we now write $X = \Spec(R)$, it follows that $R[M]$ is a positively graded 
$R$-algebra.
%Give reference for this
Therefore, by the homotopy invariance property \cite[Theorem 1.2]{Weibel89}, 
we get $KH_i(X)\cong KH_i(X[M])$ for $i \le -d$.
\end{proof}  

\section{Counterexample for non-affine schemes}\label{sec:CExm}
Recall that Weibel's $K_{-d}$-regularity conjecture is true for all Noetherian
separated schemes \cite{KST16}. We saw in the previous section that this
holds also for cancellative torsion-free semi-normal monoids. 
However, we shall show in this section 
that the monoid 
extension of $K_{-d}$-regularity conjecture 
is not valid for non-affine schemes if $M$ is not semi-normal.
This shows that the extension of Weibel's  $K_{-d}$-regularity conjecture to
monoids is a subtle question.

We let $M \subset \Z_+$ be the submonoid generated by $\{2,3\}$.
It is clear that $M$ is cancellative torsion-free, but not semi-normal.
It is also clear that the inclusion $R[M] \subset R[\Z_+]$ is same as the
inclusion $R[x^2, x^3] \subset R[x]$ for any ring $R$.
We fix a field $k$ of characteristic zero and let $C = \Spec(k[M])$.

\begin{thm}\label{thm:Counter}
Let $X$ be a connected smooth projective curve over $k$ of positive genus. 
Then the map $K_{-1}(X) \to K_{-1}(X[M])$ is not an isomorphism.
\end{thm}
\begin{proof}
Since $K_{-1}(X) = 0$, the theorem is equivalent to showing that
$K_{-1}(X[M]) \neq 0$.

We consider the conductor square
\begin{equation}\label{eqn:Counter-0}
\xymatrix@C1pc{
S \ar[r]^-{u} \ar@{=}[dr] & T \ar[r]^-{\iota'} \ar[d]^-{f'} & \A^1_k \ar[d]^-{f} 
& & S \times X \ar[r]^-{u} \ar@{=}[dr] & T  \times X \ar[r]^-{\iota'} 
\ar[d]^-{f'} &  \A^1_k  \times X \ar[d]^-{f} \\
& S \ar[r]^-{\iota} & C & & & S  \times X \ar[r]^-{\iota} & C \times X,}
\end{equation}
where the square on the right is obtained by the one on the left by the
base change via the map $X \to \Spec(k)$. The map $f$ is the normalization 
map, $S \cong \Spec(k)$ and $T \cong \Spec({k[x]}/{(x^2)})$.
The inclusion $u: S \to T$ is induced by the augmentation 
${k[x]}/{(x^2)} \surj k$.

We write $X_Y = X \times Y$ for any $k$-scheme $Y$. Note that
$X_S \cong X$. For any $n \ge 1$, we
have a commutative diagram of relative $K$-theory exact sequences:
\begin{equation}\label{eqn:Counter-0-1}
\xymatrix@C.8pc{
K_0(X_C) \ar[d] \ar[r] & K_0(n X) \ar[r] \ar[d] & K_{-1}(X_C, nX) \ar[d] 
\ar[r]^-{\alpha_n} & K_{-1}(X_C) \ar[r] & 0 \\
K_0(X_{\A^1_k}) \ar[r] & K_0(n X_T) \ar[r]^-{\beta_n} & 
K_{-1}(X_{\A^1_k}, nX_T) \ar[r] & 0, &}
\end{equation} 
where the vertical arrows are induced by $f$.
Note that for a closed immersion $W \subset Y$ defined by the sheaf
of ideals $\sI_W$, the subscheme $nW \subset Y$ is defined by $\sI^n$. 
The map $\alpha_n$ is surjective because its cokernel will otherwise map 
injectively
into $K_{-1}(nX)$. But this latter term is isomorphic to $K_{-1}(X)$
by \lemref{lem:waytored} which is zero. The map $\beta_n$ is surjective 
because $K_{-1}(X_{\A^1_k}) = 0$.

When $n =1$, we have the situation
\begin{equation}\label{eqn:Counter-0-2}
\xymatrix@C.8pc{
K_0(X) \ar[r]^-{p^*} \ar[dr]_-{\cong} & K_0(X_C) \ar[d] \ar[r]^-{\iota^*} & 
K_0(X) \ar[d]^-{f'^*} \\
& K_0(X_{\A^1_k}) \ar[r]^-{\iota'^*} & K_0(X_T) \ar@{->>}[r]^-{\beta_1} & 
K_{-1}(X_{\A^1_k}, X_T),}
\end{equation} 
where $p^*$ is induced by the projection $X_C \to X$. It follows that the
composite horizontal arrow on the top is identity. The indicated isomorphism
is by the homotopy invariance. A diagram chase shows that $\iota^*$ is 
split surjective and $K_{-1}(X_{\A^1_k}, X_T) \cong 
{\rm Coker}(\iota'^*) = {\rm Coker}(f'^*)$.
Using this, we also see that $\alpha_1$ is an isomorphism.
In particular, the theorem is equivalent to showing that 
$K_{-1}(X_C, X) \neq 0$.

We now consider the inclusion of inductive systems of closed pairs
$\{(X_{\A^1_k}, X_T)\} \inj \{(X_{\A^1_k}, nX_T)\}$.
We make the following

{\bf Claim.} The induced map of pro-abelian groups
$\{K_{-1}(X_{\A^1_k}, nX_T)\} \to K_{-1}(X_{\A^1_k}, X_T)$ is surjective.

To prove the claim, it suffices to show that the map
$K_{-1}(X_{\A^1_k}, nX_T) \to K_{-1}(X_{\A^1_k}, X_T)$ is surjective for all
$n \ge 1$.

We fix $n \ge 1$ and consider the commutative diagram 
\begin{equation}\label{eqn:Counter-0-3}
\xymatrix@C.8pc{
K_0(nX_T) \ar@{->>}[r] \ar[d] & K_{-1}(X_{\A^1_k}, nX_T) \ar[d] \\
K_0(X_T) \ar@{->>}[r] & K_{-1}(X_{\A^1_k}, X_T).}
\end{equation} 
Because $K_{-1}(X_{\A^1_k}) = 0$,  the relative $K$-theory
exact sequence tells us that the horizontal arrows are surjective.
It suffices therefore to show that the left vertical arrow is surjective.

Since $\dim(X) = 1$, it is easy to check using the Thomason-Trobaugh
spectral sequence \cite[Theorem~10.3]{TT} that there is a functorial
split exact sequence
\begin{equation}\label{eqn:Counter-0-4}
0 \to H^1(nX_T, \sO^{\times}_{nX_T}) \to K_0(nX_T) \to
H^0(X, \Z) \to 0.
\end{equation}

To compute the left term, we consider the short exact sequence of sheaves
\[
0 \to (1 + \sI_{X}) \to \sO^{\times}_{nX_T} \to \sO^{\times}_X \to 0,
\]
where $\sI_X$ is the sheaf of ideals on $nX_T$ defining $X$.
Since $nX_T \cong X \times \Spec({k[x]}/{(x^{2n})})$ and ${\rm char}(k) =0$,
this sequence is split and the exponential map 
${\rm exp}: 
\sO_X \otimes_k {(x)}/{(x^{2n})} \to (1 + \sI_{X})$ is an isomorphism of
the sheaves of abelian groups.
We thus have a commutative diagram of the
split exact sequences of sheaves of abelian groups
\begin{equation}\label{eqn:Counter-0-5}
\xymatrix@C.8pc{
0 \ar[r] & \sO_X \otimes_k {(x)}/{(x^{2n})} \ar[d] \ar[r] &
\sO^{\times}_{nX_T} \ar[r] \ar[d] &  \sO^{\times}_X \ar@{=}[d] \ar[r] &  0 \\
0 \ar[r] & \sO_X \otimes_k {(x)}/{(x^{2})} \ar[r] &
\sO^{\times}_{X_T} \ar[r] &  \sO^{\times}_X  \ar[r] &  0.}
\end{equation}

This yields an associated commutative diagram of split exact sequences
of the first cohomology groups. On the other hand, the left vertical
arrow in ~\eqref{eqn:Counter-0-5} is split surjective because
${(x)}/{(x^{2n})} \cong \stackrel{2n-1}{\underset{i = 1}\oplus} \
{(x^i)}/{(x^{i+1})}$. It follows that
the map $H^1(nX_T, \sO^{\times}_{nX_T}) \to H^1(X_T, \sO^{\times}_{X_T})$
is split surjective. Using ~\eqref{eqn:Counter-0-4}, we conclude that
the map $K_0(nX_T) \to K_0(X_T)$ is surjective. We have thus proven
the claim.

Since $\sO_X \cong \sO_X \otimes_k {(x)}/{(x^{2})}$, it also follows
from ~\eqref{eqn:Counter-0-4} and ~\eqref{eqn:Counter-0-5} that 
$K_0(X_T) \cong H^1(X, \sO_X) \oplus K_0(X)$. Using this in
~\eqref{eqn:Counter-0-2}, we get

\begin{equation}\label{eqn:Counter-0-6}
K_{-1}(X_{\A^1_k}, X_T) \cong {\rm Coker}(\iota'^*) \cong
{\rm Coker}(f'^*) \cong H^1(X, \sO_X).
\end{equation}

In the final step, we consider the commutative diagram of pro-abelian
groups

\begin{equation}\label{eqn:Counter-0-7}
\xymatrix@C.8pc{
\{K_{-1}(X_C, nX)\} \ar[r]^-{\psi} \ar[d]_-{f^*} & K_{-1}(X_C, X) \ar[d]^-{f^*} \\
\{K_{-1}(X_{\A^1_k}, nX_T)\} \ar[r]^-{\psi'} & K_{-1}(X_{\A^1_k}, X_T),}
\end{equation}
where $\psi$ is induced by the inclusion $(X_C, X) \inj (X_C, nX)$.
The arrow $\psi'$ is similarly defined. 
By the pro-cdh-descent theorem \cite[Theorem~A]{KST16}
(note that \cite[Theorem~1.1]{Krishna-2} suffices for the present case), 
applied to the
abstract blow-up square on the right of ~\eqref{eqn:Counter-0},
we see that the left vertical arrow in ~\eqref{eqn:Counter-0-7} is
an isomorphism. The bottom horizontal arrow is surjective by the above
claim. It follows that $f^* \circ \psi = \psi' \circ f^*$ is surjective.

We conclude that the map $f^*: K_{-1}(X_C, X) \to K_{-1}(X_{\A^1_k}, X_T)$
is surjective. We can now apply ~\eqref{eqn:Counter-0-6} to see that
$K_{-1}(X_C, X)$ can not be zero if the genus of $X$ is positive.
The proof of the theorem is complete.
\end{proof}

\section{$SK_0$ of monoid algebras: some reductions}
\label{sec:Redn}
Our goal in the next two sections is to prove Theorems~\ref{thm:MT-2}
and ~\ref{thm:MT-4}.
In this section, we establish some reduction steps which go into the
proof of \thmref{thm:MT-4}.

\subsection{Structure of positive cancellative
semi-normal monoids}\label{sec:SNM}
Let $M$ be a cancellative monoid with possible torsion.
We shall denote the torsion part of $\gp(M)$ by $t(M)$.
We let $\overline{M}:= {\rm Image}(M \rightarrow \gp(M)/t(M))$. 
There is an identification $\gp(M) = \gp(\overline{M}) \times t(M)$.

Suppose now that $M$ is a cancellative torsion-free monoid and let 
$\gp(M) \simeq \Z^r$.
Recall from \cite[\S~5]{Swan} that the interior of $M$ is its subset
consisting of all elements $a \in M$ such that for all $b \in M$, there
is an integer $n > 0$ and $c \in M$ such that $na = b + c$.
We denote this set by ${\rm Int}(M)$.
Note that if $M$ is generated by a finite set
$\{x_1, \ldots , x_n\}$, then $\stackrel{n}{\underset{i =1} \sum} x_i \in
{\rm Int}(M)$. In particular, ${\rm Int}(M) \neq \emptyset$.

Since $\gp(M) \simeq \Z^r$, we can view $M$ as the set of 
of integral points in the vector space $\mathbb{R}^r$. 
We let $\R_+M$ denote the set of non-negative $\R$-linear combinations
of elements in $M$. In this case, we have ${\rm Int}(M) = 
{\rm Int}(\R_+M) \cap \Z^r$, where ${\rm Int}(\R_+M)$ is the topological
interior of the cone $\R_+M$.
%If $M$ is positive, then we can choose a rational hyperplane 
%$\mathcal{H}\subset \mathbb{R}^r\setminus \{0\}$ such that 
%$\mathbb{R}_+M=\mathbb{R}_+(\mathbb{R}_+(M)\cap \mathcal{H})$. We denote
%$\Phi(L):=\mathbb{R}_+L\cap \mathcal{H}$ for any submonoid $L\subset M$.
%In this setting, ${\rm Int}(M)$ is the set of interior points of the
%polyhedron $\Phi(M)$. 

If $M$ is positive cancellative but not necessarily torsion-free, then $\ov{M}$
is a positive cancellative torsion-free monoid. We shall let 
$F(M):=\{ F_0,\ldots, F_r\}$ denote the set of faces of the cone 
$\mathbb{R}_+\overline{M}$ including $0$ and
$\mathbb{R}_+\overline M$.
We index  $F(M)$ so that
$\dim(F_i) < \dim(F_j)$ implies $i<j$ for all $F_i, F_j \in F(M)$.
Let $\rk(M) = m$  so that $\dim(\R_+ \ov{M}) = m$. 
Given two monoids $L, N$, we define
$L\leftthreetimes N:= L\times N \setminus \{(0,n)| \ n \in N \setminus \{0\}\}$.

We shall use the following description of the
semi-normalization of $M$, due to Gubeladze \cite[Lemma~9.1]{Gub-7}.

\begin{lem}\label{lem:Gub-Semi}
Let $M$ be a positive cancellative monoid (possibly with torsion).
For every $F \in F(M)$, there is a subgroup $T_F\subset t(M)$ such that 
 \begin{enumerate}
\item
$sn(M)=\cup_{F(M)} \ ({\rm Int}(n(\overline{M} \cap F)) \times T_F)$.
\item
$T_{\mathbb{R}_+\overline{M}}=t(M)$.
\item
$T_{F_1}\subset T_{F_2}$ for $F_1\subset F_2$.
\end{enumerate}
\end{lem}

Recall from \cite[\S~5]{Swan} that a submonoid $E$ of any monoid $M$ is called 
extremal if it is non-empty and
$a+b\in E \Rightarrow a, b \in E $ for all $a,b \in M$.
The following is from \cite[Theorem 5.4]{Swan}.

\begin{lem}\label{lem:extreme-crtn}
Let $M$ be a cancellative monoid. Then a submonoid $E$ of $M$
is extremal if and only if there is a monoid homomorphism 
$\phi: M \rightarrow \mathbb{N}$ such that $\phi^{-1}(0)=E.$
\end{lem}

We note down some properties of the monoid $\overline{M}$ in the following 
lemma for later use.
\begin{lem}\label{lem:M-bar-prop}
Let $M$ be a positive cancellative semi-normal monoid. Then the following hold.
\begin{enumerate}
\item
$\overline{M}$ is a positive cancellative torsion-free semi-normal monoid.
\item
Each $\overline{M}\cap F_i$ is an extremal submonoid of $\ov{M}$.
\item
Each $\overline{M}\cap F_i$ is a cancellative torsion-free semi-normal monoid.
\item	
If $N$ is a submonoid of $\overline{M}$ such that 
${\rm Int}(N) \cap (\ov{M} \cap F_i)$ is non-empty,
then $N\subset \ov{M}\cap F_i$.
\end{enumerate}
\end{lem}
\begin{proof}
We have already observed that $\overline{M}$ is a positive cancellative 
torsion-free monoid. To see semi-normality,
let $a\in gp(\overline{M})$ such that $2a, 3a \in \overline{M}$. 
Let $sn(M)$ and $sn(\ov{M})$ denote the semi-normalizations of $M$ and 
$\ov{M}$, respectively.
Using Lemma \ref{lem:Gub-Semi}, we get a commutative 
diagram
	
	\begin{equation}\label{eqn:seminormal}
	\xymatrix@C1pc{
		M \ar@{->}[r]
		\ar@{->}[d]_{\phi}
		&sn(M)= \cup_{F(M)} \ ({\rm Int}(n(\overline{M} \cap F)) 
\times T_F)
		\ar@{->}[d]^{\psi}
		\\
		\overline{M} \ar@{->}[r]
		& sn(\overline{M})=   \cup_{F(M)} \ 
({\rm Int}(n(\overline{M} \cap F)),
	}
	\end{equation}
where $\psi$ is induced from the projection on each face. Clearly, $\psi$ is a 
surjective map.

Since $\ov{M}$ is cancellative, the map $\ov{M} \rightarrow gp(\ov{M})$ is 
injective. Hence the map $\ov{M}\rightarrow sn(\ov{M})$ is injective. 
Since $a, 2a, 3a \in sn(\ov{M})$ (see \S~\ref{sec:Monoids}), 
we can lift them to $sn(M)$ using ~\eqref{eqn:seminormal}) and use the fact 
that $M=sn(M)$ to conclude that $a \in \ov{M}$. Hence, we have 
$\overline{M}=sn(\overline{M})$.

The part (2) follows from \cite[Remark~2.6(c), Exercise~2.3(a)]{BG}
(see also \cite[\S~5, Remark]{Swan}).
To prove (3), we only need to show that $\ov{M} \cap F_i$ is semi-normal.
For this, let $a\in gp(\overline{M}\cap F_i)$ such that $2a, 3a \in 
\overline{M}\cap F_i$. Since $\overline{M}$ is semi-normal, we must have
$a\in \overline{M}$. Since $F_i$ is an extremal
submonoid of $\ov{M}$ by (2) and $a+a\in  F_i$, we must have $a \in  F_i$. 
This proves (3).

For (4), let $z\in N$ and $x\in \Int(N) \cap (\ol M \cap F_i)$.
By definition of $\Int(N)$, there exists an integer $n>0$
such that $nx=z + y$ for some $y\in N$. 
Since $F_i$ is extremal and $z + y = nx \in F_i$, it follows that
$z,y \in F_i$. In particular, $z \in N \cap F_i \subset \ov{M} \cap F_i$.
\end{proof}

\begin{lem}\label{lem:monoid-ideal}
Let $M$ be a positive cancellative semi-normal monoid.
With the above notations, let 
$I_k:=\cup_{j\geq k} \ ({\rm Int}(n(\overline{M} \cap F_j)))_* \leftthreetimes 
T_{F_j}$ be subsets of $M$ for $0\leq k\leq r$.
Then $I_k$ is an ideal of $M$ for each $k$. 
\end{lem}
\begin{proof}
Since $M$ is semi-normal, Lemma~\ref{lem:M-bar-prop}(1) implies that
$\overline{M}$ is also semi-normal.
It follows from Lemma~\ref{lem:M-bar-prop}(3) that each 
$\overline{M} \cap F_j$ is also a cancellative torsion-free semi-normal
monoid. Therefore, we deduce from \cite[Lemma 6.6]{Swan} that
${\rm Int}(n(\overline{M} \cap F_j))={\rm Int}(\overline{M} \cap F_j)$.

We need to show that $I_k+M\subset I_k$ for each $k$. 
For this, we note that $({\rm Int}(\overline{M} \cap F_j))_* 
\leftthreetimes T_{F_j} =  {\rm Int}(\overline{M} \cap F_j) \times T_{F_j}$ for 
each $j$. Furthermore, $\R_+\ov{M}\cap \overline{M} = 
\amalg_j  {\rm Int}(\overline{M} \cap F_j)$ is a disjoint union
(see \cite[Lemma~5.3]{Swan}).
Hence, Lemma~\ref{lem:Gub-Semi} implies that 
$M=\amalg_j \ {\rm Int}(\overline{M} \cap F_j)\times T_F$ is a disjoint union, 
where we identify $gp(M)=gp(\overline{M})\times t(M)$ and
look at $M$ as a submonoid of $gp(M)$.

We now let $a\in I_k$ and $b\in M$.
Then $a=(a_1,a_2)\in {\rm Int}(\overline{M} \cap F_j) \times T_{F_j}$ for some 
$j \ge k$ and $b=(b_1,b_2)\in {\rm Int}(\overline{M} \cap F_i) \times T_{F_i}$ 
for some $0 \le i \le r$. 
Then $a + b=(a_1+b_1, a_2+b_2)\in M$ and hence it must belong to 
${\rm Int}(\overline{M} \cap F_{l}) \times T_{F_{l}}$ for some $l$.
We have to show that $l \ge k$ to finish the proof.

To show this, we note that
$\overline{M} \cap F_{l}$ is an extremal submonoid of $\overline{M}$.
It follows from this that $a_1, b_1\in \overline{M} \cap F_{l}$. 
In particular, we get $a_1 \in {\rm Int}(\overline{M} \cap F_j) \
\cap \ {\rm Int}(\overline{M} \cap F_{l})$.
Lemma~\ref{lem:M-bar-prop}($4$) therefore implies that
$\overline{M} \cap F_{j}\subseteq \overline{M} \cap F_{l}$.
Since each face of $\R_+\ov{M}$ is its intersection with
finitely many hyperplanes, each of which must either be non-negative or 
non-positive on $\R_+\ov{M}$, it follows that $F_j \subseteq F_l$.
This in turn implies that $l \ge j$. As $j \ge k$, we get $l \ge k$,
as desired.
\end{proof}

\subsection{Milnor square associated to positive torsion monoids}
\label{sec:Milnor-pos}
Let $R$ be a ring and $M$ a positive cancellative seminormal monoid
as in \S~\ref{sec:SNM}.
We let $A_k=R[M]/I_kR[M] = R[{M}/{I_k}]$ for $0 \le k \le r$.
There is a sequence of surjective $R$-algebra 
homomorphisms $A_{r+1}:= R[M] \surj A_r \surj \cdots \surj A_k \surj A_{k-1}
\surj \cdots \surj A_1 \xrightarrow{\simeq} A_0 = R$. 
Let $\phi_k: A_k \to A_{k-1}$ be the quotient map for $1 \le k \le r+1$.

\begin{lem}\label{lem:int-square}
The following hold for $1 \le k \le r+1$.
\begin{enumerate}
\item
${\rm Ker}(\phi_k) = R(\Int(\overline{M} \cap F_{k-1}) \times T_{F_{k-1}})$.
\item
There is a Milnor square 
\begin{equation}\label{eqn:int-sq}
\xymatrix@C1pc{
R[(\Int(\overline M \cap F_{k-1}) \times T_{F_{k-1}})_*] 
\ar@{->}[r] \ar@{->}[d]& A_k \ar@{->}[d]^{\phi_k} \\
R \ar@{->}[r]^-{\delta_{k-1}} & A_{k-1}.}
\end{equation}
\end{enumerate}
\end{lem}
\begin{proof}
Let ${I_{k-1}}/{I_{k}}$ be the image of $I_{k-1}$ in ${M}/{I_{k}}$
	under the quotient map $M \to M/{I_{k}}$. Then it is clear from
	\lemref{lem:monoid-ideal} that
	${I_{k-1}}/{I_{k}}$ is an ideal in ${M}/{I_{k}}$ and
	the monoid ${M}/{I_{k-1}}$ is obtained from ${M}/{I_{k}}$ by
	collapsing ${I_k}/{I_{k-1}}$.
	It follows that
\begin{equation}\label{eqn:Mon-exact}
	0 \to R[{I_{k-1}}/{I_{k}}] \to R[{M}/{I_{k}}] \xrightarrow{\phi_k} 
R[{M}/{I_{k-1}}] 
	\to 0
\end{equation}
	is an exact sequence of $R$-modules (see \S~\ref{sec:Monoids}).
The first assertion now follows as $R[{I_{k-1}}/{I_{k}}] =
	R[\Int(\overline{M} \cap F_{k-1}) \times T_{F_{k-1}}]$.
	
To prove (2), we first note that $\delta_k$ is the canonical inclusion
	$R \inj A_{k}$. Moreover, there is a commutative diagram
	\begin{equation}\label{eqn:int-sq-0}
	\xymatrix@C1pc{
		R \ar[dr] \ar@/_2pc/[ddr]_{\rm {Id}} 
\ar@/^2pc/[drr]^-{\delta_k} & & \\
		& R[(\Int(\overline M \cap F_{k-1}) \times T_{F_{k-1}})_*]
		\ar@{->}[r] \ar@{->}[d] & A_k \ar@{->}[d]^{\phi_k} \\
		& R \ar@{->}[r]^-{\delta_{k-1}} & A_{k-1}.}
	\end{equation}
Using this diagram, we conclude immediately from
	(1) that ~\eqref{eqn:int-sq} is Cartesian. Since $\phi_k$ is surjective,
	it follows that this is also a Milnor square. 
\end{proof}

\subsection{Reduction to positive semi-normal monoids}
\label{sec:SK-red-0}
In this subsection, we shall prove some lemmas to reduce the proof of
\thmref{thm:MT-4} to the case of positive semi-normal monoids.
We begin by recalling the definition of $SK_1$ and $SK_0$ of rings.

Let $R$ be a ring and let $H_0(R)$ denote the set of all continuous functions 
from $\Spec(R) \rightarrow \Z$ with respect to the Zariski topology on 
$\Spec(R)$ and the discrete topology on $\Z$. 
It is easy to verify that this is a ring.
There is a group homomorphism ${\rm rk}: K_0(R) \rightarrow H_0(R)$ such that
${\rm rk}([P])(\fp)$ is the rank of $P_{\fp}$ if $P$ is a projective
$R$-module.
We define $\wt{K}_0(R):= {\rm Ker}({\rm rk})$. 
There is a map ${\rm det}: K_0(R) \rightarrow \Pic(R)$ which sends $[P]$
to $[\wedge^r(P)]$, where $r$ is the rank of $P$. 
Its restriction yields a group homomorphism
$\wt{{\rm det}}: \wt{K}_0(R) \rightarrow \Pic(R)$. 
We define $SK_0(R):= {\rm Ker}(\wt{\rm det})$.
We let $SK_1(R) = {SL(R)}/{E(R)}$ so that there is a canonical 
decomposition $K_1(R) = SK_1(R) \oplus U(R)$.
The following lemma is elementary.

\begin{lem}\label{lem:SK_0-trivial}
$SK_0(R) = 0$ if and only if for every projective $R$-module $P$, one has
$P \oplus R^s \cong \wedge^r(P) \oplus R^t$, where $r = {\rm rk}(P)$.
\end{lem}
\begin{proof}
Suppose first that $SK_0(R) = 0$.
Let $P$ be a projective $R$-module of rank $r \ge 1$.
Then it is easy to see that $[P] = [\wedge^rP \oplus R^{r-1}]$ 
in $K_0(R)$. But this implies that $P\oplus R^s \cong \wedge^rP\oplus R^t$,
as is well known. The converse is obvious.
\end{proof}

The following result which connects $SK_1$ with $SK_0$, is due to 
Bass \cite[Corollary 5.12]{Bass}.
\begin{prop}\label{Bass-Milnor}
	If (\ref{eqn:MilnorSq}) is a Milnor square, then we have the 
$6$-term exact sequence
	{\small
		$$SK_1(A_1) \rightarrow SK_1(A_2)\oplus SK_1(B_1) 
\rightarrow SK_1(B_2) \rightarrow 
		SK_0(A_1) \rightarrow SK_0(A_2)\oplus SK_0(B_1) 
\rightarrow SK_0(B_2).$$} 
\end{prop}

\begin{lem}\label{Artin}
	Let $R$ be an Artinian ring and $M$ a cancellative 
torsion-free semi-normal monoid. Then the following hold.
\begin{enumerate}
\item
$SK_0(R[M])=0$.
\item
If $M$ is free and positive, then $SK_1(R[M])=0$.
\end{enumerate}
\end{lem}
\begin{proof}
By Lemma~\ref{lem:affred}, we can assume that $R$ is reduced. 
In this case, we have $R \cong k_1\times \cdots \times k_m$, where each $k_i$
is a field. The assertion (1) now follows \cite[Theorem~1.3]{Gub-4}.
 
Similar to the case of $SK_0(R)$, we have $SK_1(R[M])=SK_1(R_{\rm red}[M])$ by  
\cite[Chapter IX, Proposition 3.10, 3.11]{Bass}). So we can assume
$R$ to be a field.
Since $M \cong \Z^r_+$ by our assumption,
we get $SK_1(R[M]) \cong SK_1(R) = 0$, where the first isomorphism is
by the homotopy invariance. 
\end{proof}

Recall from \cite[\S~14]{Swan} that a ring extension $A \subset B$ is
called an elementary subintegral extension if $B = A[x]$, where
$x^2, x^3 \in A$. We say that $A \subset B$ is a subintegral extension if
it is a filtered union of elementary subintegral extensions.

\begin{lem}\label{lem:sub-int}
Let $A\subset B$ be a subintegral extension. If $SK_0(B)=0$, then $SK_0(A)=0$.
\end{lem}
\begin{proof}
By \lemref{lem:SK_0-trivial}, we need to show that 
$P \oplus A^s \cong \wedge^r(P) \oplus A^t$
for any projective $A$-module $P$ of rank $r \ge 1$.
At any rate, our assumption and  \lemref{lem:SK_0-trivial} together imply that
$P_B := P \otimes_A B$ has the property that
$P_B \oplus B^s \cong \wedge^r(P_B)\oplus B^t$.
Equivalently, we have $(P \oplus A^s)_B \cong (\wedge^r(P) \oplus A^t)_B$.
But this implies that $P \oplus A^s \cong \wedge^r(P) \oplus A^t$
by \cite[Theorem 14.1]{Swan}.
\end{proof}

\begin{lem}\label{lem:sk-pos}
Let $R$ be an Artinian ring and $M$ any monoid. 
We let $N= M\setminus U(M)$, where $U(M)$ is the group of units of 
$M$. If $SK_0(R[N_*])=0$, then $SK_0(R[M])=0$.
\end{lem}
\begin{proof}
Since $M$ is finitely generated, we have $U(M)\simeq \Z^r \oplus G$ for
some $r \ge 0$, where $G$ is a finite abelian group.
Under the hypothesis of the lemma, there is a Milnor square of rings as
in ~\eqref{eqn:Unit} (see also \cite[\S~6]{Swan}). 	
Since $R[G]$ is an Artinian ring, we have
$SK_0(R[U(M)]) \cong SK_0(R[G][\mathbb{Z}^r])=0$ by Lemma~\ref{Artin}. 
As $SK_0(R[N_*])=0$, we conclude using \propref{Bass-Milnor} 
that $SK_0(R[M])=0$.
\end{proof}

\section{$SK_0$ and the Levine--Weibel Chow group}
\label{sec:LW-**}
In this final section, we shall first prove \thmref{thm:MT-4}
and then deduce \thmref{thm:MT-2} using \thmref{thm:MT-4} and the
affine Roitman torsion theorem \cite{Krishna-1} 
for the Levine-Weibel Chow group.

\subsection{$SK_0$ of cancellative monoid algebras}\label{sec:Canc-mon}
We shall first prove our main result for the vanishing of $SK_0$ of 
monoids algebras when the underlying monoid is cancellative.
More precisely, we prove:

\begin{lem}\label{lem:sk}
Let $R$ be an Artinian ring and $M$ a commutative cancellative 
monoid. Then $SK_0(R[M])=0$.
\end{lem}
\begin{proof}
We can assume, using Lemmas~\ref{lem:sub-int} and ~\ref{lem:sk-pos}, 
that $M$ is semi-normal and positive.
If we further assume that $M$ is torsion-free, then 
$SK_0(R[M])=0$ by Lemma~\ref{Artin}. So we can assume that 
$M$ is positive, cancellative and semi-normal but not torsion-free.

We can write $M=sn(M)=\cup_{F(M)}(\Int(\overline{M} \cap F) \times T_F))$
by \lemref{lem:Gub-Semi}. Note that $\Int(n(\overline{M} \cap F)) =
\Int(\overline{M} \cap F)$ by \cite[Lemma~6.6]{Swan}.
We shall prove the lemma by applying \propref{Bass-Milnor} to
	~\eqref{eqn:int-sq}. We note that $A_0 \cong A_1 \cong R$ and 
$A_{r+1} = R[M]$.
Hence, it suffices to show using an induction argument that
$SK_0(A_k) = 0$ for  $1 \le k \le r+1$.

Since  $SK_0(R) = 0$,  the base case for induction is established.
It suffices now to prove the lemma for $A_k$ assuming it holds for 
$A_{k-1}$ for $k \ge 2$.
Using \lemref{lem:int-square} and \propref{Bass-Milnor},
it suffices to show that 
$SK_0(R[(\Int(\overline M \cap F_{k-1}) \times T_{F_{k-1}})_*]) = 0$. 

Now, we observe that $R[(\Int(\overline M \cap F_{k-1}) \times T_{F_{k-1}})_*]
= R[(\Int(\overline M \cap F_{k-1}))_* \leftthreetimes T_{F_{k-1}}]$.
Letting $L = (\Int(\overline M \cap F_{k-1}))_*$, we see that
$L$ is a normal positive cancellative torsion-free monoid
(see \cite[Proposition~2.40]{BG}).
Writing $F_{k-1} = \stackrel{m}{\underset{i = 1}\prod} {\Z}/{n_i}$, it
suffices to show by induction on $m$ that
\begin{equation}\label{eqn:main-tor-free} 
SK_0(R[L\leftthreetimes (\stackrel{m}{\underset{i = 1}\prod} {\Z}/{n_i})]) = 0. 
\end{equation}

If $m = 0$, then ~\eqref{eqn:main-tor-free} is true by our assumption.  
	In general, there is a Milnor square (see 
	\cite[Proof of Theorem~1.1, p.~214]{Gub-7})
	\begin{equation}\label{eqn:main-tor-free-0} 
	\xymatrix@C1pc{
		\Lambda \ar[r] \ar[d] & R[L\leftthreetimes 
		(\stackrel{m}{\underset{i = 1}\prod} {\Z}/{n_i})] \ar[d] \\
		R[\Z_+][L \leftthreetimes 
		(\stackrel{m-1}{\underset{i = 1}\prod} {\Z}/{n_i})] 
\ar@{->>}[r]^-{\pi} &
		R[{\Z}/{n_m}][L \leftthreetimes 
		(\stackrel{m-1}{\underset{i = 1}\prod} {\Z}/{n_i})],}
	\end{equation}
	where
	\begin{enumerate}
		\item
		$\Lambda = A + B$,
		\item
		$A = R[{\Z}/{n_m}][(L \leftthreetimes \Z_+) \leftthreetimes
		(\stackrel{m-1}{\underset{i = 1}\prod} {\Z}/{n_i})]$,
		\item
		$B = R[t^{n_m}, t^{n_m+1}-t, \cdots , t^{2n_m-1}- t^{n_m-1}] 
\subset R[t] \simeq
		R[\Z_+]$ and
		\item
		$\pi$ is induced by $t \mapsto x$ for some generator 
$x \in {\Z}/{n_m}$.
	\end{enumerate}
By induction on $m$ and \propref{Bass-Milnor}, it suffices to show that
$SK_0(\Lambda) = 0$.

We now consider another Milnor square (see 
\cite[Proof of Theorem~1.1]{Gub-7})
\[
\xymatrix@C1pc{
A \ar[r] \ar[d] & \Lambda \ar[d] \\
R \ar[r] & B.}
\]
Using this square and \propref{Bass-Milnor} again,
it suffices to prove that $SK_0(A) = 0 = SK_0(B)$.

Since $L \leftthreetimes \Z_+$ is a finitely generated positive, cancellative 
torsion-free normal monoid, it follows by induction on $m$ that
$SK_0(A) = 0$.  To prove the result for $B$, we can assume $R$ is reduced
by \lemref{lem:affred}.
We can further assume that $R$ is a field. We now observe that 
there is a conductor square
	\[
	\xymatrix{
		B\ar@{->}[r]
		\ar@{->}[d]
		&R[t]
		\ar@{->}[d]
		\\
		B/C \ar@{->}[r]
		&R[t]/C,     
	}
	\]
where $C$ is the conductor ideal of the extension $B\subset R[t]$. 
Being a subring of the integral domain $R[t]$,
$B$ is an integral domain. 
Since $t(t^{n_m}-1)\in C$ is a non-zero divisor, we see that the height of 
$C$ is positive. 	
It follows that $B/C$ and $R[t]/C$ are both Artinian rings.
In particular, $SK_1(R[t]/C) = 0$.
Since $SK_0(B/C) = SK_0(R[t]) = 0$, it follows from
\propref{Bass-Milnor} that $SK_0(B) = 0$.
\end{proof}

\subsection{The final step for \thmref{thm:MT-4}}\label{sec:MT-4**}
Let $M$ be an arbitrary (finitely generated) monoid. Recall from
\cite[\S~15]{Swan} that an ideal $P \subset M$ is called prime if
$P \neq M$ and if $x,y \in M, \ x+y \in P$ implies $x \in P$ or $y \in P$.
This is equivalent to saying that $N = M \setminus P$ is a non-empty submonoid
of $M$. In this case, there are monoid algebra morphisms
$R[N] \to R[M] \to R[N] \cong {R[M]}/{R[P]}$ for any ring $R$ whose
composite is identity.
It is also easy to check that if $\fp$ is a prime ideal of $R[M]$, them
$\fp \cap M$ is a prime ideal of $M$.

An ideal $I \subset M$ is called a radical ideal if every element
$x \in M$ with the property $nx \in I$ for some $n > 0$, belongs to $I$. 
If $I \subset M$ is an ideal, we let $\sqrt{I} = \{x \in M|nx \in I \
\mbox{for \ some} \ n > 0\}$. We call this the radical of $I$. 
It is easy to check that $\sqrt{I}$ is a radical ideal of $M$.
Our key step to finish the proof of \thmref{thm:MT-4} is
the following result which 
generalizes \cite[Lemma~15.6]{Swan} to arbitrary monoids.

\begin{lem}\label{lem:primary}
Let $I \subset M$ be a proper radical ideal in a monoid $M$. Then
there are prime ideals $\fp_1, \ldots , \fp_r$ in $M$ such that
$I = \stackrel{r}{\underset{i =1}\cap} \fp_i$.
\end{lem}
\begin{proof}
In this proof, we shall use the multiplicative notation for the
monoid operation on $M$.
Since $\Z[I]$ is an ideal of the Noetherian ring $\Z[M]$ (note that $M$
is finitely generated), we have the (irredundant) primary decomposition
$\Z[I] = \stackrel{r}{\underset{i =1}\cap} Q_i$. We let $P_i$ denote the
unique associated prime of $Q_i$ in $\Z[M]$. Recall that if we write
the monoid operation of $M$ multiplicatively, then there
is a multiplicative monoid embedding $M \inj \Z[M]$ which sends $m$ to
$(1 \cdot m)$. We let $\fq_i = Q_i \cap M$
and $\fp_i = P_i \cap M$ via this embedding. 
It is easy to check from the definition of $\Z[M]$ that $I = \Z[I] \cap M$.
We therefore 
get $I = \Z[I] \cap M = \stackrel{r}{\underset{i =1}\cap} (Q_i \cap M) =
\stackrel{r}{\underset{i =1}\cap} \fq_i$.
In particular, we have $I = \stackrel{r}{\underset{i =1}\cap} \fq_i \subset
\stackrel{r}{\underset{i =1}\cap} \fp_i$.
Note that each $\fp_i$ is a prime ideal of $M$ as we already observed
earlier.

Suppose now that there is an element $x \in M$ which lies in
$\stackrel{r}{\underset{i =1}\cap} \fp_i$.
This implies that $x$ lies in $\stackrel{r}{\underset{i =1}\cap} P_i$ inside
$\Z[M]$. Since this intersection is same as $\sqrt{\Z[I]}$
in $\Z[M]$, it follows that $x^m \in \Z[I]$ for some $m > 0$.
Since $x \in M$ and $M$ is multiplicatively closed in $\Z[M]$, 
we also have $x^m \in M$. Consequently, we get $x^m \in \Z[I] \cap M = I$.
But $I$ is a radical ideal of $M$ and hence we must have $x \in I$.
We have therefore shown that $\stackrel{r}{\underset{i =1}\cap} \fp_i 
\subset I$. This proves the lemma.
\end{proof}

The following lemma is elementary.
\begin{lem}\label{lem:ELM-monoid}
Let $M$ be any monoid and $I \subset M$ any ideal. Then 
$R[\sqrt{I}] \subseteq \sqrt{R[I]}$ for any ring $R$.
\end{lem}
\begin{proof}
As in the proof of \lemref{lem:primary},
we shall use the multiplicative notation for the
monoid structure of $M$. Let $u = a_1x_1 + \cdots + a_rx_r \in R[\sqrt{I}]$
with $a_i \in R$ and $x_i \in \sqrt{I}$ for each $i$.
Since $I \subset M$ is an ideal, we can find $m_0 \gg 0$ such that
$x^m_i \in I$ for all $m \ge m_0$ and all $i \ge 1$.
It is then straightforward to check using the multinomial expansion of $u^m$
that for all $m \ge rm_0$, all $M$-coefficients of $u^m$ will lie in $I$.
That is, if we write $u^m = b_1y_1 + \cdots + b_sy_s$, then each $y_i \in I$
for $1 \le i \le s$. But this implies that $u \in \sqrt{R[I]}$.
\end{proof}

\vskip .3cm

{\bf {Proof of \thmref{thm:MT-4}:}}
We can assume $R$ to be reduced by \lemref{lem:affred}.
We can then further assume that $R$ is a field.
Let $N$ be a cancellative monoid and $I \subset N$ an ideal such that 
$M = N/I$. If $I = N$, then $R[M] \cong R$ and the theorem is 
obvious. So we can assume that $I \subset N$ is a proper ideal.
In particular, $I \cap U(N) = \emptyset$.

Since the image of $R[\sqrt{I}]$ in $R[M]$ is nilpotent 
by \lemref{lem:ELM-monoid}, we can assume that
$I$ is a radical ideal of $N$ by \lemref{lem:affred}.
In this case, \lemref{lem:primary} says that we can write 
$I = \stackrel{r}{\underset{i =1}\cap} \fp_i$, where 
$\fp_1, \ldots , \fp_r$ are prime ideals in $N$.

We shall now prove the theorem by induction on $r \ge 1$.
If $r = 1$, then $I$ is a prime ideal of $N$ so that $J = N \setminus I$
is a submonoid of $N$ and $R[M] \cong R[J]$. Since $N$ is cancellative,
it follows that $J$ is also cancellative. So we are done in this case
by \lemref{lem:sk}.

In general, we let $\fp = \fp_1$ and $\fq = \fp_2 \cap \cdots \cap \fp_r$
with $r \ge 2$. Then $I = \fp \cap \fq$ and $L = \fp \cup \fq$ is a proper
ideal of $N$ since $\fp, \fq \subset N \setminus U(N)$.
Note here that the union of two ideals in a ring is generally not an
ideal, but it is true for ideals in a monoid.

Since ${\fp}/I \xrightarrow{\cong} {L}/{\fq}$, it is easy to check
using exact sequences of the type ~\eqref{eqn:Mon-exact} that the diagram
\begin{equation}\label{eqn:MT-4-0}
\xymatrix@C.8pc{
{R[N]}/{R[I]} \ar[r] \ar[d] & {R[N]}/{R[\fq]} \ar[d] \\
{R[N]}/{R[\fp]} \ar[r] & {R[N]}/{R[L]}}
\end{equation}
is Cartesian (see \cite[Proof of Theorem~15.1]{Swan}).
Furthermore, if we let $N' = N \setminus \fp$ and $S = N' \cap L =
N' \cap \fq$, then $ {R[N']}/{R[S]} \xrightarrow{\cong} {R[N]}/{R[L]}$.
On the other hand, the inclusion $N' \inj N$ takes $S$ into $\fq$ and induces 
a map ${R[N']}/{R[S]} \to {R[N]}/{R[\fq]}$. This shows that the right vertical
arrow in ~\eqref{eqn:MT-4-0} is a split surjection.

By induction on $r$, we see that $SK_0({R[N]}/{R[\fp]}) = 0 =
SK_0({R[N]}/{R[\fq]})$. It follows from \propref{Bass-Milnor} that
the sequence
\[
SK_1({R[N]}/{R[\fp]}) \oplus SK_1({R[N]}/{R[\fq]}) \to
SK_1({R[N]}/{R[L]}) \to SK_0({R[N]}/{R[I]}) \to 0
\]
is exact. Since the right vertical arrow in ~\eqref{eqn:MT-4-0} is a split 
surjective homomorphism of $R$-algebras, it follows that the first arrow from
left in this
exact sequence in surjective. But this implies that
$SK_0(R[M]) = SK_0({R[N]}/{R[I]}) = 0$. This finishes the proof of the
theorem.
$\hfill \square$

\vskip .3cm

\subsection{The Levine-Weibel Chow group}\label{sec:LW*}
Our goal in this subsection is to prove \thmref{thm:MT-2}.
As we explained in Introduction,  % \S~\ref{sec:Intro}, 
our approach is to use the
vanishing of $SK_0$ of the given monoid algebra
and then use the affine Roitman 
torsion theorem for the Levine-Weibel Chow group from \cite{Krishna-1}.
Before we give the details, we recall the
definition of the Levine-Weibel Chow group from \cite{LW} 
for reader's reference.

Let $k$ be an algebraically closed field of any characteristic.
Let $A$ be a finite type reduced $k$-algebra and let $X = \Spec(A)$ denote
the spectrum of $A$.  We shall say that a point $x \in X$ is regular
if $\sO_{X,x}$ is a regular local ring. We let $X_{\rm reg} \subset X$
denote the regular locus of $X$ so that $x \in X_{\rm reg}$ if and only if
it is a regular point. We let $X_{\rm sing} = X \setminus X_{\rm reg}$ 
denote the singular locus of $X$. A closed subscheme $C \subset X$ is
called a Cartier curve if it is a scheme of pure dimension one
such that the following hold.
\begin{enumerate}
\item
No irreducible component of $C$ lies in $X_{\rm sing}$.
\item
For every $x \in C \cap X_{\rm sing}$, the ideal $I_{C,x}$ of
$C$ in the local ring 
$\sO_{X,x}$ is generated by a regular sequence.
\end{enumerate}

For a Cartier curve $C$, let $k(C,X_{\rm sing})^{\times}$ denote the group of
invertible elements in the ring of total quotients of $C$
which are regular along $C \cap X_{\rm sing}$.

Let $\sZ_0(X)$ denote the free abelian group on the
set of regular closed points of $X$. Given a Cartier curve $C \subset X$
and $f \in k(C,X_{\rm sing})^{\times}$, we have the divisor $\divf_C(f)$ of $f$
in the sense of \cite[\S~1.2]{Fulton} (see also \cite[\S~1]{LW}). 
Since $f$ is regular and
invertible along $X_{\rm sing}$, it follows that $\divf_C(f) \in \sZ_0(X)$.
We let $\CH^{LW}_0(X)$ be the quotient of $\sZ_0(X)$ by the subgroup
$\sR_0(X)$, generated by $\divf_C(f)$, where $C$ runs over all
Cartier curves on $X$ and $f \in k(C,X_{\rm sing})^{\times}$.
We shall use the notations $\CH^{LW}_0(X)$ and $\CH^{LW}_0(A)$ interchangeably. 

When $X$ is regular, $\CH^{LW}_0(X)$ coincides with the classical Chow group of
0-cycles on $X$ (see \cite[Chapter~1]{Fulton}).
This is however not the case when $X$ has singularity.
In this case, it is $\CH^{LW}_0(X)$ which is known to be the correct
Chow group of 0-cycles and is supposed to constitute the 0-cycle part
of the conjectural
full theory of cohomological Chow groups of $X$. Furthermore, it is
directly related to theory of vector bundles on $X$ unlike the classical
homological Chow group $\CH_0(X)$. 

Since the structure sheaf of a regular closed point $x \in X$ has finite
tor-dimension over $X$, it follows that this point has a class $cyc(x)$
in $K_0(X)$. We thus get a cycle class map $cyc: \sZ_0(X) \to K_0(X)$.
Furthermore, it is shown in \cite[Proposition~2.1]{LW} that this map
kills $\sR_0(X)$ so that there is a well-defined cycle class map
\begin{equation}\label{eqn:cycle-class}
cyc: \CH^{LW}_0(X) \to K_0(X).
\end{equation}

We have the following result about this cycle class map which is supposed to
be well known.

\begin{lem}\label{lem:cycle-class-*}
Suppose that $\dim(X) \ge 2$. Then the cycle class map has the factorization
\[
cyc: \CH^{LW}_0(X) \to SK_0(X).
\]
\end{lem}
\begin{proof}
Let $x \in X$ be a regular closed point and let $U = \Spec(X) \setminus \{x\}$.
We then have the Thomason-Trobaugh localization exact sequence
\begin{equation}\label{eqn:cycle-class-*-0}
K_0(\{x\}) \xrightarrow{\cong} K^{\{x\}}_0(X) \to K_0(X) \to K_0(U).
\end{equation}

Since the image of the middle arrow is the subgroup generated by 
$cyc(x)$, it suffices to show that the maps
$H^0(X, \Z) \to H^0(U, \Z)$ and $\Pic(X) \to \Pic(U)$ are injective.
The first assertion is obvious. So we need to prove the second assertion.
For this, we let $S = \{x\}$. Using the isomorphism
$H^1(X, \sO^{\times}_X) \cong \Pic(X)$ and the exact sequence
\[
H^1_S(X,  \sO^{\times}_X) \to H^1(X, \sO^{\times}_X) \to H^1(U, \sO^{\times}_U),
\]
it suffices to show that $H^1_S(X,  \sO^{\times}_X) = 0$.
Since $x$ is a regular closed point, we have
$H^1_S(X,  \sO^{\times}_X) \cong H^1_S(X_{\rm reg},  \sO^{\times}_{X_{\rm reg}})$
by excision. So we can assume $X$ is regular.

In this case, we have a long exact sequence
\[
H^0(X, \sO^{\times}_X) \to H^0(U, \sO^{\times}_U) \to
H^1_S(X,  \sO^{\times}_X) \to H^1(X, \sO^{\times}_X) \to H^1(U, \sO^{\times}_U).
\]

Since $X$ is regular and codimension of $S$ is at least two in $X$, 
it is well known that the map
$H^i(X, \sO^{\times}_X) \to H^i(U, \sO^{\times}_U)$ is an isomorphism for
$i \le 1$. We are therefore done.
\end{proof}

The second key ingredient in the proof of \thmref{thm:MT-2} is the following
affine Roitman torsion theorem for 0-cycles. 
This is an old conjecture of Murthy 
\cite{Murthy} and is now a theorem \cite[Corollary~7.6]{Krishna-1}.

\begin{thm}\label{thm:Murthy}
Let $A$ be a reduced affine algebra over an algebraically closed field $k$. 
Then the cycle class map $cyc: \CH^{LW}_0(A) \to K_0(A)$ is injective.
\end{thm}

\vskip .3cm

{\bf {Proof of \thmref{thm:MT-2}:}}
We let $X = \Spec(k[M])$. Using \thmref{thm:MT-4} and 
\lemref{lem:cycle-class-*}, it suffices to show that the map
$cyc: \CH^{LW}_0(X) \to K_0(X)$ is injective. But this follows from 
\thmref{thm:Murthy}.
$\hfill \square$

\subsection{A different proof of  \thmref{thm:MT-4} for
pctf  monoids}\label{sec:pctf-sk}
We end our discussion with another proof of \thmref{thm:MT-4}
in the special case in which the underlying monoid is (partially cancellative
and) torsion-free.
Note that \thmref{thm:MT-4} proves the vanishing of $SK_0(R[M])$ for
all partially cancellative monoids which are not necessarily torsion-free.
But we decided to include this different proof in the special case
because it is more $K$-theoretic in nature and crucially uses
negative $K$-theory. Our hope is that this $K$-theoretic approach may
be helpful in future generalizations of \thmref{thm:MT-4}
to more general monoid algebras.

\vskip .3cm

%{\bf {Proof of \thmref{thm:MT-4} for pctf  monoids:}}

We now begin the proof. We let $M$ be a partially cancellative
torsion-free monoid and $R$ an Artinian ring.
We want to show that $SK_0(R[M]) = 0$.

If $M$ is cancellative, then the result
follows from \lemref{lem:sk}. We can therefore assume that
$M$ is torsion-free but only partially cancellative.  
As in the proof of \lemref{lem:AffMainVans}, we can assume that
$M = N/I$, where $N$ is a cancellative torsion-free monoid.
Associated to the Milnor-square ~\eqref{pc-Milnor}, there exists
a commutative diagram
\begin{equation}\label{eqn:PC-SK}
\xymatrix@C.8pc{
SK_0(R) \oplus SK_0(R[N]) \ar[r]^-{\alpha} \ar[d] &  SK_0(R[M]) \ar[d] & \\
K_0(R) \oplus K_0(R[N]) \ar[r]^-{\beta} &  K_0(R[M]) \ar[r] & K_{-1}(R[I_*]),}
\end{equation}
where the bottom row is exact by Proposition~\ref{prop:Milnor*}. 
The vertical arrows are clearly injective.
Here, the arrows $\alpha$ and $\beta$ are the canonical maps.

Since $R$ is Artinian and
$I_*$ is a cancellative torsion-free monoid,
it follows from Lemmas~\ref{lem:affred} and ~\ref{lem:reg-neg} that
the last term of the bottom exact sequence in ~\eqref{eqn:PC-SK} vanishes.
%It follows from \lemref{lem:sk} that the left term of the top exact sequence
%vanishes. 
A diagram chase shows that that $SK_0(R[M])$ lies in  the image of
$\beta$. Since $\wt{K}_0(R)$ is a canonical direct summand of $K_0(R)$,
it follows that the image of $SK_0(R[M])$ in $\wt{K}_0(R[M])$
lies in the image of the map
$\wt{K}_0(R) \oplus \wt{K}_0(R[N]) \to \wt{K}_0(R[M])$.

We now consider the commutative diagram
\begin{equation}\label{eqn:SK-0-*}
\xymatrix@C.8pc{
& SK_0(R) \oplus SK_0(R[N]) \ar[r]^-{\alpha} \ar[d] &  SK_0(R[M]) \ar[d] & \\
\wt{K}_0(R[I_*]) \ar[r] \ar@{->>}[d] & 
\wt{K}_0(R) \oplus \wt{K}_0(R[N]) \ar[r]^-{\beta} \ar@{->>}[d] &  
\wt{K}_0(R[M]) \ar[r] \ar@{->>}[d] & 0 \\
\Pic(R[I_*]) \ar[r] & \Pic(R) \oplus \Pic(R[N]) \ar[r] & \Pic(R[M]) \ar[r] & 0.}
\end{equation}

If $A$ is a ring and $L$ is a projective $A$-module of rank one,
then $[L] - [A] \in \wt{K}_0(A)$. Furthermore, the map
${\rm det}: \wt{K}_0(A) \to \Pic(A)$ sends $[L] - [A]$ to $[L]$. 
It follows that the map $\wt{K}_0(A) \to \Pic(A)$ is surjective. 
We conclude that the lower vertical arrows in ~\eqref{eqn:SK-0-*} 
are all surjective.
Since the two columns on the right are exact and the two lower rows are
also exact, a diagram chase shows that $\alpha$ is surjective.
On the other hand, $SK_0(R) \oplus SK_0(R[N]) = 0$ by
\lemref{lem:sk}. It follows that $SK_0(R[M]) = 0$.
This finishes the proof.
$\hfill \square$

\vskip .4cm

%\bigskip

\noindent\emph{Acknowledgments.} 
This project was conceived and a part of the work was done when the authors 
were visiting Hausdorff Institute for Mathematics (HIM), Bonn as part of
institute's $K$-theory program in 2017. The authors thank the
institute and the organizers of the program for invitation and support.
The second author would like to thank Tata Institute and University of Warwick
for invitation and support. He would also like to thank
Govt. of India SERB fellowship for financial support during his stay at
Warwick. This paper was greatly inspired by several works of
Gubeladze, Corti{\~n}as, Haesemeyer, Walker and Weibel on $K$-theory of
monoid algebras. The authors would like to thank the referee for carefully reading
the paper and providing many valuable suggestions to improve its presentation.

\end{document}